\pdfoutput=1
\documentclass[a4paper, 11pt]{style/preprint}
\usepackage[T1]{fontenc}
\usepackage[utf8]{inputenc}
\usepackage[left=1.5in, right=1.5in, bottom=1.5in]{geometry}
\usepackage{style/comment}
\usepackage{style/mhenvs}
\usepackage{style/mhsymb}
\usepackage{style/trees}

\usepackage[full]{textcomp}
\usepackage[osf]{newtxtext}
\usepackage{amssymb}
\usepackage{amsmath}
\usepackage{mathtools}
\usepackage{mhequ}
\usepackage{booktabs}
\usepackage{tikz}
\usepackage{mathrsfs}
\usepackage{longtable}
\usepackage{microtype}
\usepackage{wasysym}
\usepackage{centernot}
\usepackage{enumitem}
\usepackage{orcidlink}
\usepackage[dvipsnames]{xcolor}
\usepackage{hyperref}
\usepackage{crossreftools}
\usepackage{fontawesome5}
\usepackage{physics}
\usepackage{shuffle}

\DeclareMathOperator{\mul}{mult}
\DeclareMathOperator{\Hom}{Hom}
\DeclareMathOperator{\Osc}{Osc}

\mathtoolsset{showonlyrefs=true}

\newif\ifdark
\IfFileExists{dark}{\darktrue}{\darkfalse}

\ifdark

\definecolor{darkred}{rgb}{0.9,0.2,0.2}
\definecolor{darkblue}{rgb}{0.7,0.3,1}
\definecolor{darkgreen}{rgb}{0.1,0.9,0.1}
\definecolor{pagebackground}{rgb}{.15,.21,.18}
\definecolor{pageforeground}{rgb}{.84,.84,.85}
\pagecolor{pagebackground}
\AtBeginDocument{\globalcolor{pageforeground}}

\else

\definecolor{darkred}{rgb}{0.7,0.1,0.1}
\definecolor{darkblue}{rgb}{0.4,0.1,0.8}
\definecolor{darkgreen}{rgb}{0.1,0.7,0.1}
\definecolor{pagebackground}{rgb}{1,1,1}
\definecolor{pageforeground}{rgb}{0,0,0}

\fi

\definecolor{newred}{RGB}{208,16,76}
\definecolor{newgreen}{RGB}{34,125,81}

\makeatletter
\newcommand{\globalcolor}[1]{%
	\color{#1}\global\let\default@color\current@color
}
\makeatother

\DeclareSymbolFont{timesoperators}{T1}{ptm}{m}{n}
\SetSymbolFont{timesoperators}{bold}{T1}{ptm}{b}{n}

\makeatletter
\renewcommand{\operator@font}{\mathgroup\symtimesoperators}
\makeatother

\DeclareMathAlphabet{\mathbbm}{U}{bbm}{m}{n}
\DeclareFontFamily{U}{BOONDOX-calo}{\skewchar\font=45 }
\DeclareFontShape{U}{BOONDOX-calo}{m}{n}{
	<-> s*[1.05] BOONDOX-r-calo}{}
\DeclareFontShape{U}{BOONDOX-calo}{b}{n}{
	<-> s*[1.05] BOONDOX-b-calo}{}
\DeclareMathAlphabet{\mcb}{U}{BOONDOX-calo}{m}{n}
\SetMathAlphabet{\mcb}{bold}{U}{BOONDOX-calo}{b}{n}

\makeatletter 

\newcommand*{\fat}{}
\DeclareRobustCommand*{\fat}{%
	\mathbin{\mathpalette\bigcdot@{}}}
\newcommand*{\bigcdot@scalefactor}{.5}
\newcommand*{\bigcdot@widthfactor}{1.15}
\newcommand*{\bigcdot@}[2]{%
	\sbox0{$#1\vcenter{}$}
	\sbox2{$#1\cdot\m@th$}%
	\hbox to \bigcdot@widthfactor\wd2{%
		\hfil
		\raise\ht0\hbox{%
			\scalebox{\bigcdot@scalefactor}{%
				\lower\ht0\hbox{$#1\bullet\m@th$}%
			}%
		}%
		\hfil
	}%
}

\DeclareRobustCommand{\TitleEquation}[2]{\texorpdfstring{\StrLeft{\f@series}{1}[\@firstchar]$\if b\@firstchar\boldsymbol{#1}\else#1\fi$}{#2}}

\makeatother

\theoremnumbering{Alph}
\renewtheorem{theorem*}{Theorem}
\newtheorem{assumption}[lemma]{Assumption}
\newtheorem{example}[lemma]{Example}
\numberwithin{equation}{section}

\def\slash{\leavevmode\unskip\kern0.18em/\penalty\exhyphenpenalty\kern0.18em}
\def\dash{\leavevmode\unskip\kern0.18em--\penalty\exhyphenpenalty\kern0.18em}

\let\epsilon\varepsilon

\def\${|\!|\!|}

\setlist{noitemsep,topsep=4pt}
\def\para_#1{/\!\!/_{\!#1}}

\newcommand\blfootnote[1]{%
  \begingroup
  \renewcommand\thefootnote{}\footnote{#1}%
  \addtocounter{footnote}{-1}%
  \endgroup
}

\hfuzz=\maxdimen
\vfuzz=\maxdimen
\begin{document}

\title{Non-explosion principles for branched rough differential equations with unbounded coefficients}
\author{Kexing Ying \orcidlink{0000-0002-8292-4746}}
\institute{Institute of Mathematics, EPFL, Switzerland}

\maketitle

\begin{abstract}
  Expanding upon the work of \cite{Li:25}, we provide a non-explosion criterion for branched 
  rough differential equations (bRDEs) where we continue to allow for the drift and the rough 
  coefficient to be unbounded in the branched setting. Moreover, by providing a characterization of
  ``pure area'' branched rough paths -- branched rough paths over the zero path, we provide two 
  different constructions of bRDEs which explode in finite time, and thus demonstrating the 
  sharpness of the criterion. Finally, by realizing a trade-off between the growth of the 
  coefficient of the bRDE and the decay of its higher-order derivatives, we provide a new non-explosion 
  principle for bRDEs which allows for the coefficient to grow even faster than what is provided 
  in \cite{Li:25}.
\end{abstract}

\blfootnote{The author would like to thank Xue-Mei Li for their supervision and helpful discussions. 
  The author acknowledges support from the Swiss National Science Foundation project MINT (10000849)
  and from NCCR SwissMAP.}

\setcounter{tocdepth}{2}
\tableofcontents

\noindent

\section{Introduction}

The theory of rough paths, introduced in the seminal works of Lyons \cite{Lyons:94, Lyons:98}, provides a path-wise framework for
differential equations driven by signals that are too irregular for classical calculus 
(say having H\"older regularity \(\gamma < \frac{1}{2}\)). 
It restores the convergence of the Riemann sum by enriching the driving signal with 
an extra component corresponding to the second-order iterated integral. 
This viewpoint was demonstrated to be an especially powerful tool in stochastic analysis as it 
provides a deterministic formulation of equations driven by stochastic processes, 
e.g. (fractional) Brownian motions. Moreover, it restores the continuity of the
solution map with respect to the driving signal and the component corresponding to the iterated integral. 
We note the standard references on rough path theory \cite{Friz:10} and \cite{Friz:20} while we also 
remark on some of its successful applications in multiscale systems (e.g. \cite{Chevyrev:19, Hairer:22, Gehringer:21, Li:25b}),
(S)PDEs (e.g. \cite{Diehl:15, Deya:19, Hofmanova:19, Li:26}), and in the context of machine learning 
(e.g. \cite{Chevyrev:26, Kidger:19, Morrill:21}) just to name a few.

For paths with H\"older regularity \(\gamma < \frac{1}{3}\), extra care is needed and additional components are required to ensure the 
convergence of the Riemann sum. In this case, the rough paths most commonly considered are the so-called 
geometric rough paths, which are namely the rough paths whose integration theory
satisfies the usual chain rule and are enriched with components corresponding to the 
higher-order iterated integrals. These rough paths arise naturally as limits of
smooth paths and are well-suited for many applications. Nevertheless, e.g. in the context of 
finance, where It\^o integration is preferred over Stratonovich for the payoff of a continuous-time strategy,
as the latter can be heuristically described as ``seeing into the future'', 
the geometric rough path formalism is not necessarily always desirable. The search for the correct notion 
of rough paths in this setting led to the introduction of branched rough paths by Gubinelli in 
\cite{Gubinelli:04, Gubinelli:06} (see also \cite{Cass:16} in the infinite-dimensional case).
We also note the articles \cite{Tapia:20, Varzaneh:25} which contain a nice overview of branched rough paths.
In this case, the components enriching the path no longer just correspond to the (linear) iterated integrals, 
but also integrals of the products (of the integral of products, and so on). This is in contrast to the geometric case, 
in which these components are simply reduced to the iterated integrals via the shuffle relations. Instead, in the 
absence of the chain rule, these data are naturally encoded by the Hopf algebra of rooted trees 
(namely the Connes-Kreimer Hopf algebra, cf. \cite{Hoffman:03, Oudom:05}). 

We in this article study non-explosion for rough differential equations (RDEs) driven by these branched 
rough paths. More precisely, we consider branched rough differential equations (bRDEs) of the form
\begin{equation}\label{eq:bRDE-intro}
  \dd \mathbf{Y}_t = b(\mathbf{Y}_t) \, \dd t + \phi(\mathbf{Y}_t) \, \dd \mathbf{X}_t,
\end{equation}
where \(\mathbf{X}\) is a branched rough path and the coefficients are allowed to be unbounded. 
For geometric RDEs, global existence of RDEs with unbounded coefficients 
has been studied by many authors, including but not limited to the works of 
\cite{Davie:07, Bailleul:12, Driver:18, Bailleul-Catellier} and also \cite{Riedel:16} in which the latter established 
non-explosion of \eqref{eq:bRDE-intro} assuming \(b\) having linear growth and \(\phi\) bounded. 
For RDEs driven by non-geometric rough paths in the \(\gamma \in \left(\frac{1}{3}, \frac{1}{2}\right]\) regime, 
we also note the works \cite{Lejay:09, Lejay:12} which showed non-explosion of the RDE while assuming 
that the coefficient has bounded derivative. We also mention the works of \cite{Fang:07, Scheutzow:17, Anzeletti:23,Anzeletti:24} 
regarding strong completeness and path-by-path 
well-posedness of SDEs which are closely related to the non-explosion problem for RDEs when the driving signal is the It\^o rough path lift 
of a Brownian motion. In the context of bRDEs, we note \cite{Bonnefoi:22} in which the authors 
studied the ``coming down from infinity'' property of a bRDE with a cubic drift 
(albeit in this case, the drift has a damping effect rather than a growth effect). 

The present work is mainly motivated by the work of Li with the author \cite{Li:25} and extends 
the non-explosion principles presented there to the branched setting. In particular, we provide a
non-explosion criterion for bRDEs with drift for which we observe that a trade-off can be made 
between the growth of the drift coefficient \(b\) and the growth of the rough coefficient \(\phi\).
Moreover, we obtain a new non-explosion principle for bRDEs by observing that a trade-off can be 
made between the growth of the coefficient \(\phi\) and the decay of its higher-order derivatives. 
Finally, by providing a characterization of branched rough paths over the zero path, we provide 
two different constructions of bRDEs which explode in finite time, demonstrating the sharpness of the non-explosion 
criterion.

We remark that the non-explosion criterion we provide is independent of the choice of the rough
path lift. Thus, by transforming the higher order components of the branched rough path appropriately, one 
might achieve a better growth condition than what is provided directly by the criterion. The example of comparing  
the Stratonovich lift against the It\^o lift of a Brownian motion already illustrates this as one 
version could have a better growth condition than the other. Namely, the correction term could 
provide cancellation or could worsen the growth condition. Consequently, while one can generally 
write a bRDE as a geometric RDE on an extended space via the Hairer-Kelly embedding \cite{Hairer:14}
(cf. also \cite{Boedihardjo:19}),
the criterion one obtains after applying this embedding is most certainly not optimal.  
More commentary on this phenomenon can also be observed in the examples of Section~\ref{sec:sharpness}.

Finally, we also note that working in the branched setting is also motivated by the similarities it bears 
to analogous controlled-type descriptions used in the theory of singular SPDEs \cite{Hairer:14b}. 
In those theories, the same Hopf algebra structure appears as a consequence of 
the Picard iterations of Duhamel's formula. We note the article \cite{Bruned:19} 
which provides a detailed analysis of the algebraic structure of branched rough paths and Hairer's 
theory of regularity structures. One example of interest is the recent work \cite{Gerencser:26}, 
in which the authors discussed ``a pure area model'', i.e. a non-trivial lift over the trivial noise. 
This is the analogue of the pure area rough path in the context of regularity structures and so, 
one reasonably expects that the discussion in Section~\ref{subsec:pure-area} on pure area branched rough paths 
could be extended to this context. In this light, one may hope that the technique developed 
here could be furthered to the setting of singular SPDEs, in which the non-explosion problem is of 
even greater interest.

\subsection{Strategy for non-explosion}
We outline the strategy used to show non-explosion of an RDE. 
For simplicity, in this section we assume that we work at level \(N = 2\) and that the 
rough differential equation has no drift, i.e. we consider \(Y : [0, \xi) \to \mathbb{R}^m\) the 
maximal solution to the RDE
\begin{equation}\label{eq:rde-strat}
  \dd Y_t = \phi(Y_t) \dd \mathbf{X}_t, 
\end{equation}
where \(\mathbf{X} \in \CC^\gamma([0, T], \mathbb{R}^d)\) is a \(\gamma\)-H\"older rough path with 
\(\gamma \in (\frac{1}{3}, \frac{1}{2}]\) and \(\phi : \mathbb{R}^m \to \CL(\mathbb{R}^d, \mathbb{R}^m)\).

The proof strategy for showing non-explosion of \((Y_t)\) essentially boils down to the following simple 
observation. Suppose \((r_n)_{n \in \mathbb{N}} \subseteq \mathbb{R}_+\) is a non-decreasing sequence such that 
\(r_n \to \infty\) as \(n \to \infty\). We keep track of the location of \(Y\) by introducing the times
\(\tau^{r_n} = \tau^n = \inf\{t : \|Y_t\| \ge r_n\}\) and \(\sigma^{r_n} = \sigma^n = \sup\{t < \tau^{n + 1} : \|Y_t\| \le r_n\}\) 
so that \(\|Y_t\| \in [r_n, r_{n + 1}]\) for all \(t \in [\sigma^n, \tau^{n + 1})\). Then, for 
any \(t > 0\), if 
\begin{equation}\label{eq:tau-ineq}
  \|Y_{(\sigma^n + t) \wedge \tau^{n + 1}} - Y_{\sigma^n}\| < r_{n + 1} - r_n
\end{equation}
for some \(n \in \mathbb{N}\), we have that \(t < \tau^{n + 1} - \sigma^n\). 
Consequently, if we can find a sequence \((t_n)_{n \in \mathbb{N}}\) such that for each \(n \in \mathbb{N}\),
Equation~\eqref{eq:tau-ineq} holds for \(t = t_n\), then 
\[\xi = \lim_{n \to \infty} \tau^{n} \ge \sum_{n = N}^\infty (\tau^{{n + 1}} - \tau^{n}) \ge 
  \sum_{n = N}^\infty (\tau^{n + 1} - \sigma^{n}) > \sum_{n = N}^\infty t_n.\]
Thus, we may conclude that the RDE has non-explosion if \((t_n)\) can be chosen to be non-summable.

A natural contender for the contour sequence \((r_n)\) is to take \(r_n = Rn\) for some constant 
\(R > 0\). In this case, denoting the integral \(\int_0^t \phi(Y_s) \dd \mathbf{X}_s =: Z_t\), we may observe
the simple estimate
\begin{equation}\label{eq:del-Y-est}
  \begin{split}
    \|Y_{(\sigma^{n} + t) \wedge \tau^{n + 1}} - Y_{\sigma^{n}}\| 
    & = \|Z_{(\sigma^n + t) \wedge \tau^{n + 1}} - Z_{\sigma^{n}}\|\\
    & \le \|R^Z_{\sigma^{n}, (\sigma^{n} + t) \wedge \tau^{n + 1}}\| 
      + \|\phi(Y_{\sigma^{n}})\| \|X\|_{\gamma} t^{\gamma}
  \end{split}
\end{equation}
with \(R^Z_{s, t} := Z_t - Z_s - \phi(Y_s) (X_t - X_s)\), i.e. it is the remainder term in the 
construction of the rough integral in the context of controlled rough paths.
Thus, assuming a sub-linear growth condition on \(\phi\) of the form 
\(\|\phi(y)\| \lesssim \|y\|^{\gamma}\), by taking \(t_n = \epsilon n^{-1}\) for some 
small \(\epsilon\) independent of \(n\), the term \(\|\phi(Y_{\sigma^{n}})\| \|X\|_{\gamma} t^{\gamma}\) 
can be made uniformly (in \(n\)) arbitrarily small. On the other hand, under certain conditions, it can be shown 
that the H\"older norm of the remainder \(R^Z\) can be controlled on sufficiently small intervals 
(cf. Corollary~\ref{cor:h-tilde-est}). To be more precise, for \(t < r_n^{-1}\), we have the estimate
\[\|R^Z_{\sigma^n, (\sigma^n + t) \wedge \tau^{n + 1}}\| 
  \le \|R^{Z}\|_{2\gamma; I} t^{2\gamma} 
  \lesssim (R(n + 1))^{2\gamma} t^{2\gamma}\]
where \(I := [\sigma^n, (\sigma^n + t) \wedge \tau^{n + 1}]\).
Hence, for \(t_n\) chosen above (shrinking \(\epsilon\) if necessary),
the remainder term can also be made uniformly arbitrarily small. Thus, choosing \(\epsilon\) so that 
both terms are strictly less than \(\frac{1}{2} R\) provides the desired estimate. 

In the presence of a drift, the above argument can be adapted (cf. \cite[Lemma 3.5 and Theorem 4.12]{Li:25}) 
and a similar conclusion holds for the solution of the RDE under the same growth conditions on \(\phi\).
To be \textit{slightly} more precise, by an analogous argument, one can control the growth of the 
solution of a perturbed ODE of the form
\[\dd Y_t = b(Y_t) \dd t + \dd \gamma_t\]
provided one has sufficient control over the perturbation \(\gamma\). With this, we can bootstrap 
the previous control by considering instead the system 
\[\begin{cases}
  \dd Y_t & = b(Y_t) \dd t + \dd \eta_t\\
  \dd \eta_t & = \phi(Y_t) \dd \mathbf{X}_t
  \end{cases}.\]

We in this article generalize this result to the setting of branched RDEs resulting in 
Theorem~\ref{thm:non-explosion}. For this, we provide an a priori estimates on the higher-order 
Gubinelli derivatives of the solution in Section~\ref{sec:apriori} from which we deduce the aforementioned 
estimates on the remainder term \(R^Z\) in this lower regularity setting (Corollary~\ref{cor:h-tilde-est}). 
Moreover, in proving this remainder estimate, one quickly realizes that 
the bound remains to hold while allowing \(\phi\) to have faster growth should its derivatives 
have sufficient decay. This naturally leads to the question of whether one can still show a similar 
non-explosion criterion when \(\phi\) has (almost) linear growth. To achieve this, we need to adapt the above argument.
Indeed, while the \(R^Z\) term satisfies a similar estimate, the term 
\(\|\phi(Y_{\sigma^n})\| \|X\|_{\gamma} t^{\gamma} \sim n t^\gamma\) in~\eqref{eq:del-Y-est} can no 
longer be uniformly bounded without making \((t_n)\) summable. 
Instead, we will leverage a different set of contours \((r_n)\) which 
grows exponentially. To be more precise, taking \(r_n \sim 2^n\), we have that 
\(r_{n + 1} - r_n \sim 2^n\) and thus, the desired estimate is of the form
\begin{equation}\label{eq:del-Y-est2}
  \|Y_{(\tau^{{n}} + t) \wedge \tau^{{n + 1}}} - Y_{\tau^{n}}\| < \epsilon 2^n
\end{equation}
for some small \(\epsilon\) independent of \(n\). Namely, for this choice of \((r_n)\), 
the term \(\|\phi(Y_{\tau^{n}})\| \|X\|_{\gamma} t^{\gamma} \sim 2^n t^\gamma\) in \eqref{eq:del-Y-est2} is no longer an issue as 
one can simply take \(t_n \sim \epsilon^{\frac{1}{\gamma}}\) which is clearly not summable. 

However, one should not yet celebrate as now the \(R^Z\) term becomes problematic as the 
estimate there (cf. Lemma~\ref{lem:R-est}) would impose that 
\(t_n \lesssim 2^{-n}\) which is most certainly summable. To remedy this, we introduce \(2^n\) sub-level bands 
within each exponential level set of width 1 and track the times at which the solution crosses these bands.
As it turns out, the remainder \(R^Z\) can be controlled on these smaller intervals while cancellation 
of terms occurs for the Riemann sum estimate on \(R^Z\) on these intervals. This is the content of 
Section~\ref{sec:beyond} and Appendix~\ref{sec:appendixA} resulting in Theorem~\ref{thm:non-explosion2}.

\subsection{Notations}
We collect here only the conventions used repeatedly throughout the paper and which are not 
introduced formally in the preliminaries. For a two-parameter map $A$, we write
\[\delta A_{st} := A_t - A_s,\]
and for an interval $I \subseteq [0, T]$ and exponent $\alpha > 0$, we use the H\"older semi-norms
\[\|A\|_{\alpha; I} := \sup_{s < t \in I} \frac{\|A_{st}\|}{|t - s|^{\alpha}},
	\qquad \|A\|_{\infty; I} := \sup_{t \in I} \|A_t\|.\]
As usual, $A \lesssim B$ means that $A \le C B$ for a constant $C > 0$ independent of the relevant variables.

In the non-explosion arguments of Section~\ref{sec:non-exp}, for $r > 0$ we write
\(\tau^r := \inf\{t \ge 0 : \|Y_t\| \ge r\}\).
Moreover, in Section~\ref{sec:beyond} and the appendix, where the argument 
is one-dimensional, we instead use
\[\tilde\tau^r := \inf\{t : |Y_t| \ge 2r\},
	\qquad \tilde\sigma^r := \sup\{t < \tilde\tau^r : |Y_t| \le r\}.\]
We introduce any further proof-local notations at the point of use.

\section{Preliminaries}
We dedicate this section to a brief self-contained introduction to the setting of branched rough paths. We also 
note the works of \cite{Hairer:14, Bonnefoi:22, Tapia:20, Varzaneh:25, Bruned:19} and of course, \cite{Gubinelli:06}
all of which include nice expositions to this theory.

\subsection{Rooted Labelled Trees}

For higher order rough paths, it is useful to introduce tree notations to keep track of the 
algebraic and analytic conditions for the extra terms arising from the Taylor expansion.

\begin{definition}
  Denote \(\CT\) for the set of all rooted and labeled trees and \(\CF\) for the set of the associated 
  forests, i.e. \(\CF\) is the free abelian monoid generated by \(\CT\) and \(f \in \CF\) is a finite 
  collection of rooted labeled trees (where we allow multiple copies of the same tree).

  Notation wise, we use greek letters \(\tau, \sigma, \rho, \dots\) for trees in \(\CT\) and 
  \(f, g, h \dots\) for forests in \(\CF\).
\end{definition}

Some examples of rooted and labeled trees are 
\[\tree<X>, \tree<1>, \tree<2>, \tree<10>, \tree<3>, \tree<20>, \tree<100>, \cdots\]
We write \(f = \prod_{i \in I} \tau_i \in \CF\) as a commutative product where \(\tau_i \in \CT\) and \(I\) 
is a finite index set. We allow \(\CT\) to contain the empty tree and we denote it by \(\mathbf{1}\) 
since it is the multiplicative identity in \(\CF\). We introduce the rooting operation \([-]_a : \CF \to \CT\) where 
for \(f \in \CF\), we define \([f]_a\) for the tree obtained by joining all the trees in \(f\) to the
root \(a\), e.g. 
\[[\mathbf{1}]_a = \tree<X> \text{ and } [\tree<X> \tree<1cb>]_d = \tree<X10>.\]
For \(\tau \in \CT\), denote \(f_\tau\) for the number of copies of \(\tau\) in \(f\). Thus, we can write 
\(f = \prod_{\tau \in \CT} \tau^{f_\tau}\)
where \(f_\tau\) is zero for all but finitely many terms. Namely, \(\tau \in f\) if and only if 
\(f_\tau \neq 0\). For \(\tau \in \CT\), the order of \(\tau\): \(|\tau|\) is simply the number of vertices of 
\(\tau\) and for \(f \in \CF\), \(|f| = \sum_{\tau \in \CT} f_\tau |\tau|\). We also write \(\#f = \sum_{\tau \in \CT} f_\tau\).

We denote \(\< \CF \>\) for the free vector space generated by \(\CF\) so \(\< \CF\>\) 
is a unital commutative algebra with identity \(\mathbf{1}\). We define the inner product of \(\< \CF \>\)
by setting 
\begin{equation}\label{eq:inner-product}
  \< f, f'\> = \mathbf{1}_{f = f'}
\end{equation}
for all \(f, f' \in \CF\) and extending linearly. 

We define the coproduct \(\Delta : \< \CF \> \to \< \CF \> \otimes \< \CF \>\) 
by first defining \(\Delta\) on \(\CF\) inductively and then extending linearly. In particular, we set
\begin{itemize}
  \item \(\Delta \mathbf{1} = \mathbf{1} \otimes \mathbf{1}\).
  \item For \(f \in \CF\), \(\Delta f = \prod_{h \in \CT} (\Delta h)^{f_h}\).
  \item For \(f \in \CF\), \(\Delta [f]_a = [f]_a \otimes \mathbf{1} + (\text{id} \otimes [-]_a) \Delta f\).
\end{itemize} 
We provide some examples:
\begin{itemize}
  \item \(\Delta \tree<X> = \tree<X> \otimes \mathbf{1} + \mathbf{1} \otimes \tree<X>\).
  \item \(\Delta \tree<1> = \tree<1> \otimes \mathbf{1} + \tree<Xb> \otimes \tree<X> + \mathbf{1} \otimes \tree<1>\).
  \item \(\Delta \tree<2> = \tree<2> \otimes \mathbf{1} + \tree<Xb> \tree<Xc> \otimes \tree<X> + \tree<Xc> \otimes \tree<1> 
    + \tree<Xb> \otimes \tree<1ac> + \mathbf{1} \otimes \tree<2>\).
\end{itemize}
Geometrically, the coproduct of a tree corresponds to the sum of all the ways of splitting the tree
into two parts by a single cut, with the part containing the root placed on the right hand side of the tensor product.
Namely, we have that 
\begin{equation}\label{eq:coproduct-sum-subtrees}
  \Delta \tau = \tau \otimes \mathbf{1} + \mathbf{1} \otimes \tau + 
      \sum_{\sigma \subsetneq \tau} (\tau \setminus \sigma) \otimes \sigma
\end{equation}
where the sum is over all proper subtrees \(\sigma\) of \(\tau\) with the same root and 
\(\tau \setminus \sigma\) is the rooted tree obtained by removing \(\sigma\) from \(\tau\).

\subsection{Hopf Algebra}

The algebra \(\< \CF\>\) equipped with the coproduct \(\Delta\) forms what is known as a 
Hopf algebra (in particular, it is known as the Connes-Kreimer Hopf algebra). We in this section provide an 
informal description of this structure.

Before we proceed with defining Hopf algebras, we first need to introduce \textit{bialgebras}.
Suppose we have an algebra \(\CH^*\) acting on another algebra \(\CH\) via the action 
\[\< \cdot, \cdot \> : \CH^* \otimes \CH \to \mathbb{R}.\]
Roughly speaking, a bialgebra is then \(\CH\) equipped with a \textit{coproduct} \(\Delta : \CH \to \CH \otimes \CH\) 
which encodes this pairing and in particular, preserves the product structure of its \textit{coalgebra} \(\CH^*\). More precisely,
\(\Delta\) is dual to the product operator \(\nabla^*\) of \(\CH^*\) in the sense that 
\[\< \nabla^*(f \otimes g), h\> = \< f \otimes g, \Delta h\>\]
where \(\nabla^* : \CH^* \otimes \CH^* \to \CH^*\) is such that \(\nabla^*(x, y) = xy\). Thus, in some 
sense, once we have the coproduct \(\Delta\), we can forget about \(\CH^*\) and work solely with \(\CH\).

A bialgebra \((\CH, \Delta)\) is \textit{graded} if it has the decomposition 
\[\CH = \bigoplus_{n \in \mathbb{N}} \CH_{(n)}\]
where \(\CH_{(n)}\) are vector spaces such that for all \(n, m \in \mathbb{N}\),
\[\CH_{(n)} \cdot \CH_{(m)} \subseteq \CH_{(n + m)}, \text{ and } 
  \Delta \CH_{(n)} \subseteq \bigoplus_{p + q = n} \CH_{(p)} \otimes \CH_{(q)}.\]
Denoting \(\CF_{(n)} = \{f \in \CF : |f| = n\}\), taking \(\CH_{(n)} = \< \CF_{(n)}\>\), 
we observe that \(\CF\) is a graded bialgebra. 

While the coproduct preserves the product structure of \(\CH^*\), we would also like a map which 
preserves the units (invertible elements) of \(\CH^*\). This is achieved with the \textit{antipole} 
map \(S : \CH \to \CH\) which is required to satisfy
\begin{equation}\label{eq:antipole}
  \nabla (S \otimes \text{id}) \Delta h = \nabla (\text{id} \otimes S) \Delta h = \< \mathbf{1}, h\> \mathbf{1}
\end{equation}
for all \(h \in \CH\), where \(\nabla\) is the product operator of \(\CH\) and 
\[\<\mathbf{1}, h\> = 
\begin{cases}
  1 & \text{if } h = \mathbf{1},\\
  0 & \text{otherwise}.
\end{cases}\]
This condition is motivated by the fact that, if \(S\) is such that for any unit \(f \in \CH^*\), 
\(S^* f = f^{-1}\) where \(S^* : \CH^* \to \CH^*\) is the dual of \(S\), then
\[\< f \otimes f, (S \otimes \text{id})\Delta h\> 
  = \< \nabla^* (S^* f \otimes f), h\> = \< S^* f \cdot f, h\> = \< \mathbf{1}, h\>.\]
On the other hand, as \(f\) is a unit,
\[\< f \otimes f, \< \mathbf{1}, h\> \mathbf{1} \otimes \mathbf{1}\> = \< \mathbf{1}, h\>.\]
Hence, in order for \(S\) to preserve the unit structure of \(\CH^*\), it is necessary to require
\(\nabla (S \otimes \text{id})\Delta h = \nabla \< \mathbf{1}, h\> \mathbf{1} \otimes \mathbf{1} 
= \< \mathbf{1}, h\> \mathbf{1}\). 

With the antipole motivated, a Hopf algebra is then simply a bialgebra equipped with an antipole.

\begin{proposition}
  A graded bialgebra satisfying \(\CH_{(0)} = \mathbb{R}\) is automatically a Hopf algebra.
\end{proposition}

Thus, as \(\< \CF_{(0)}\> = \mathbb{R}\), it follows \(\< \CF\>\) is a Hopf algebra.
The antipole \(S\) of \(\< \CF\>\) can be computed explicitly by using the identity~\eqref{eq:antipole}. 
We give some examples:
\begin{itemize}
  \item \(S \mathbf{1} = \mathbf{1}\).
  \item \(0 = \nabla (S \otimes \text{id}) \Delta \tree<X> = S \tree<X> + \tree<X>\). Thus, \(S \tree<X> = -\tree<X>\).
  \item \(0 = \nabla (S \otimes \text{id}) \Delta \tree<1> = S \tree<1> + (S \tree<Xb>) \tree<X> + \tree<1>\). 
    Thus, \(S \tree<1> = -\tree<1> + \tree<X> \tree<Xb>\).
  \item \(0 = \nabla (S \otimes \text{id}) \Delta \tree<2> = S \tree<2> + (S \tree<Xb> \tree<Xc>) \tree<X> + (S \tree<Xc>) \tree<1> + (S \tree<Xb>) \tree<1ac> + \tree<2>\).
    Thus, \(S \tree<2> = -\tree<2> + \tree<Xc> \tree<1> + \tree<Xb> \tree<1ac> - \tree<X> \tree<Xb> \tree<Xc>\).
\end{itemize}

Equipping \(\< \CF \>\) with the inner product as defined by Equation~\eqref{eq:inner-product}, 
we can view \(\< \CF \>\) acting on itself via the action \(\< f, \cdot \>\) 
for all \(f \in \< \CF \>\). Thus, the underlying bialgebra of \((\< \CF \>, \Delta, S)\) 
can be viewed as \(\CH = \< \CF \>\) with itself as its own coalgebra \(\CH^* = \< \CF \>\) 
(with a different multiplication). We now give a description of the algebraic structure of 
\(\CH^* = \< \CF \>\).

By the property of the coproduct, if \(\CH^* = \< \CF \>\) has product \(\nabla^*(f, g) = f \curvearrowright g\), then
for any \(f, g, h \in \CF\), we have the defining property for \(\nabla^*\):
\[\< f \curvearrowright g, h\> = \< \nabla^*(f, g), h\> = \< f \otimes g, \Delta h\>.\]
The product \(\curvearrowright\) is known as the \textit{convolution product} (or the graphing operation e.g. \cite{Bonnefoi:22}). 
We observe that for \(f, g \in \CF\), 
\(\< f \curvearrowright g, h\> = 1\) if and only if \(f \otimes g\) is a term of \(\Delta h\). For trees, the 
convolution can be described geometrically with
\[\tau_1 \curvearrowright \tau_2 = \tau_1 \tau_2 + \tau_1 \curvearrowright_t \tau_2\]
for \(\tau_1, \tau_2 \in \CT\) where \(\tau_1 \curvearrowright_t \tau_2\) denotes the sum of all trees obtained 
by attaching \(\tau_1\) to a vertex of \(\tau_2\). 

On the other hand, \(\< \CF \>\) equipped with concatenation \(\cdot\) (i.e. free product defined by \(h_1 \cdot h_2 = h_1 h_2\)) 
forms an associative (\textit{non-commutative}) algebra. 
We denote its coproduct by \(\delta :  \< \CF \> \to  \< \CF \> \otimes  \< \CF \>\), i.e. we have that
\[\< \delta f, h_1 \otimes h_2\> = \< f, h_1 h_2\>.\]
Thus, if \(\tau \in \CT\), then, using Sweedler's notation,
\[1 = \< \tau, \tau \cdot \mathbf{1}\> = \< \delta \tau, \tau \otimes \mathbf{1}\>
  = \sum \< \tau^{(1)}, \tau\> \< \tau^{(2)}, \mathbf{1}\>.\]
Namely, there exists a unique component \(i\) of \(\delta \tau\) for which \(\tau^{(1)i} = \tau\) and 
\(\tau^{(2)i} = \mathbf{1}\).

The tensor algebra \(T(\mathbb{R}^d)\) can be identified in \(\< \CF \>\) by identifying 
\(e_a\) by \(\tree<X>\), \(e_b \otimes e_a\) by \(\tree<1>\), \(e_c \otimes e_b \otimes e_a\) by \(\tree<10>\)
and so on for \(a, b, c = 1, \dots, d\). Namely, denoting 
\(\iota : T(\mathbb{R}^d) \hookrightarrow \< \CF \>\) for this inclusion, 
\(\iota(T(\mathbb{R}^d))\) corresponds to the set of all trees with a single branch.

\begin{definition}
  Let \(\CH\) be a Hopf algebra with coalgebra \(\CH^*\). Denote \(\Hom(\CH, \mathbb{R}) \subseteq \CH^*\) 
  for the space of characters (i.e. \(\mathbb{R}\)-valued algebra homomorphisms) of \(\CH\).
\end{definition}

For all (by an abuse of notation) \(f = \< f, \cdot \> \in \Hom(\CH, \mathbb{R})\), 
as \(f\) is multiplicative, it follows that 
\[\< f, h_1 h_2\> = \< f, h_1\> \< f, h_2\> = 
  \< f \otimes f, h_1 \otimes h_2\>.\]
On the other hand, denoting \(\delta\) for the coproduct on the coalgebra \(\CH^*\), we have 
\[\< \delta f, h_1 \otimes h_2\> = \< f, h_1 h_2\>.\] 
Thus, we can identify elements of \(\Hom(\CH, \mathbb{R})\) with \(f \in \CH^*\) which satisfies 
\(\delta f = f \otimes f\). More precisely, 
\[\Hom(\CH, \mathbb{R}) \simeq \{f \in \CH^* : \delta f = f \otimes f\} =: G(\CH).\]

\begin{definition}
  \(G(\CH)\) forms a group and is known as the \textit{Butcher group}.
\end{definition}

\begin{proposition}
  For all \(f \in G(\CH)\), \(f^{-1} = S^* f\) with \(S^*\) being the adjoint of the antipole.
\end{proposition}

\subsection{Branched Rough Paths}

Finally, we can now define branched rough paths. Defining \(\CH^*_{(n)} = \< \CF_{(n)}\>\) 
where the notation \(^*\) indicates we are working in the dual space, we define 
\[G_N(\CH) = \frac{G(\CH)}{\{f \in G(\CH) : \< f, h\> = 0 \text{ for all } h \in \CF, \ 1 \le |h| \le N\}}\]
where the fraction indicates the quotient group (it is not difficult to see that the denominator 
is indeed a normal subgroup). Namely, we identify two characters which agree on all forests of order at most \(N\). Then, for a path \(X = (X^a)_{a \in I}\), we define the branched 
rough paths over \(X\) by specifying its iterated integrals.

\begin{definition}
  For a map \(\mathbf{X} : [0, T] \to G_N(\CH)\), we denote \(\mathbf{X}_{st} = \mathbf{X}^{-1}_s \curvearrowright \mathbf{X}_t\).
  For any \(\gamma \in (\frac{1}{N + 1}, \frac{1}{N}]\), we say \(\mathbf{X}\) is a \(\gamma\)-H\"older branched rough path if 
  \[\sup_{s < t \in [0, T]} \frac{|\< \mathbf{X}_{st}, f\>|} {|t - s|^{\gamma|f|}} < \infty\]
  for all \(f \in \CF_N\).

  For a branched rough path \(\mathbf{X}\), we define its path component \(X^a\) by 
  \[\delta X^a_{st} = \< \mathbf{X}_{st}, \tree<X>\>.\]
\end{definition}

We observe that the path components of a \(\gamma\)-H\"older branched rough path are \(\gamma\)-H\"older. 
Moreover, Chen's condition (that is \(\mathbf{X}_{st} = \mathbf{X}_{su} \curvearrowright \mathbf{X}_{ut}\)) is automatically 
satisfied by branched rough paths. 

For general trees \(\tau = [f]_a\), we interpret \(\< \mathbf{X}_{st}, \tau\> = 
  \int_s^t \< \mathbf{X}_{sr}, f\> \dd X^a_r\) (note that here, the right hand side is strictly 
formal and is simply a notation for the left hand side). Thus, 
\[\< \mathbf{X}_{st}, \tree<X> \> = \int_s^t \dd X^a_r = \delta X^a_{st}, \qquad
  \< \mathbf{X}_{st}, \tree<1> \> = \int_s^t \int_s^r \dd X^b_u \dd X^a_r,\]
\[\< \mathbf{X}_{st}, \tree<2>\> = \int_s^t \left(\int_s^r \dd X^b_u\right) \left(\int_s^r \dd X^c_u\right) \dd X^a_r,\] 
and so on. These form a basis should we repeatedly apply Taylor expansion of a solution to an RDE.

\begin{definition}
  A path \(\mathbf{Y} : [0, T] \to \CH_{N - 1} = \bigoplus_{n = 0}^{N - 1} \CH_{(n)}\) is a 
  \(\mathbf{X}\)-controlled rough path for some \(\gamma\)-H\"older branched rough path \(\mathbf{X}\) if 
  \begin{equation}\label{eq:remainder}
    R^{Y, f}_{st} = R^f_{st} = \< f, \mathbf{Y}_t\> - \< \mathbf{X}_{st} \curvearrowright f, \mathbf{Y}_s\>
  \end{equation}
  satisfies 
  \[\sup_{s < t \in [0, T]} \frac{|R^f_{st}|}{|t - s|^{(N - |f|)\gamma}} < \infty\] 
  for all \(f \in \CF_{N - 1}\). We denote \(Y_t = \< \mathbf{1}, \mathbf{Y}_t\>\) for the 
  path component of \(\mathbf{Y}\) and furthermore, for \(f \in \CF_{N - 1}\), denote 
  \(\partial^f_X Y_t = \< f, \mathbf{Y}_t\>\). We call \(\partial^f_X Y_t\) the \(f\)-th Gubinelli 
  derivative of \(\mathbf{Y}\) against \(\mathbf{X}\).
\end{definition}

For simplicity, our discussion from now on focuses to the case where \(I = \{1, 2, \dots, d\}\) 
and \(X = (X^a)_{a = 1}^d = (\<\mathbf{X}, \tree<X>\>)_{a = 1}^d\) is a path in \(\mathbb{R}^d\).
Moreover, for computations involving generic labels \(a \in I\), we omit the label from the notations 
whenever it is clear from the context. 

It is easy to see that this definition of \(\mathbf{X}\)-controlled rough paths (at level \(N = 2\)) 
corresponds to the usual definition of controlled rough paths of H\"older regularity 
\(\gamma \in (\frac{1}{3}, \frac{1}{2}]\). Indeed, testing the remainder against \(h = \mathbf{1}\), we obtain
\[R^{\mathbf{1}}_{st} = Y_t - \< \mathbf{X}_{st} \curvearrowright \mathbf{1}, \mathbf{Y}_s\> 
  = Y_t - \< \mathbf{X}_{st}, \mathbf{Y}_s\>.\]
Now, as 
\begin{align*}
  \< \mathbf{X}_{st}, \mathbf{Y}_s\> 
  & = \sum_{f \in \CF_{1}} \< f, \mathbf{Y}_s\> \< \mathbf{X}_{st}, f\>
    = Y_s \< \mathbf{X}_{st}, \mathbf{1}\> 
      + \< \tree<X'>, \mathbf{Y}_s\> \< \mathbf{X}_{st}, \tree<X'>\>\\
  & = Y_s + \< \tree<X'>, \mathbf{Y}_s\> \delta X_{st}
\end{align*}
where the last equality follows as \(\mathbf{X}_{st} \in \Hom(\CH, \mathbb{R})\) so 
\(\< \mathbf{X}_{st}, \mathbf{1}\> = 1\). Thus, the Gubinelli derivative of \(Y\) against 
\(X\) in this case is \(Y'_s = \partial^{\tree<X'>}_X Y_s = \< \tree<X'>, \mathbf{Y}_s\>\). 
More generally, for any \(N\), we find that
\begin{align*}
  R^{\mathbf{1}}_{st} 
  & = \delta Y_{st} - 
  \sum_{f \in \CF_{N - 1} \setminus \{\mathbf{1}\}} \partial_X^f Y_s \< \mathbf{X}_{st}, f\>\\
  & = \delta Y_{st} - Y'_s \delta X_{st} - 
    \sum_{f \in \CF_{N - 1} \setminus \{\mathbf{1}, \tree<X'>\}} \< f, \mathbf{Y}_s\> \< \mathbf{X}_{st}, f\>
\end{align*}
Moreover, we observe that a similar formula holds for other levels of the remainder (Proposition~\ref{prop:remainder}). 
For the sake of computation, we introduce the following short hands. For all \(f \in \CF_{N}\), 
\begin{itemize}
  \item \(\Delta_r f = \Delta f - \mathbf{1} \otimes f\) which is known as the reduced coproduct,
  \item for \(g, h \in \CF_{N}\), set \(c_r(f, g, h) = \< g \otimes h, \Delta_r f\> = 
      \<g \curvearrowright h, f\>\mathbf{1}_{g \neq \mathbf{1}}\). 
    Namely, 
    \[\Delta_r f = \sum_{g, h \in \CF_{N}} c_r(f, g, h) g \otimes h.\]
\end{itemize}
\begin{lemma}\label{lem:c'-path}
  For \(f \in \CF_{N}\), \(c_r(f, g, h) = \sum_l \<f, l\> c_r(l, g, h)\).
\end{lemma}
\begin{proof}
  \begin{equs}
    \sum_{g, h} \left(\sum_l \<f, l\> c_r(l, g, h)\right) g \otimes h
    & = \sum_{l} \<f, l\> \sum_{g, h} c_r(l, g, h) g \otimes h\\
      = \sum_{l} \<f, l\> \Delta_r l & = \Delta_r \left(\sum_l \<f, l\> l\right) = \Delta_r f.
  \end{equs}
\end{proof}
Using these notations, we have the equivalent formulation of the remainder term.
\begin{proposition}\label{prop:remainder}
  For a \(\gamma\)-H\"older rough path \(\mathbf{X}\) and a path \(\mathbf{Y} : [0, T] \to \CH^*_{N - 1}\), 
  for all \(f \in \CF_{N - 1}\), 
  \[R_{st}^f = \delta (\partial^f_X Y)_{st} - 
    \sum_{g, h \in \CF_{N - 1}} c_r(g, h, f) \partial^g_X Y_s \< \mathbf{X}_{st}, h\>\]
  where \(R^f\) is defined as in Equation~\eqref{eq:remainder}.
\end{proposition}
\begin{proof}
  We compute, 
  \begin{align*}
    R^f_{st} & = \partial^f_X Y_t - \< \mathbf{X}_{st} \curvearrowright f, \mathbf{Y}_s\>
        = \partial^f_X Y_t - \<\mathbf{X}_{st} \otimes f, \Delta\mathbf{Y}_s\>\\[1ex]
      & = \partial^f_X Y_t - \<\mathbf{X}_{st} \otimes f, \mathbf{1} \otimes \mathbf{Y}_s\> 
          - \<\mathbf{X}_{st} \otimes f, \Delta_r\mathbf{Y}_s\>\\[1ex]
      & = \delta (\partial^f_X Y)_{st} - \sum_{h, l} c_r(\mathbf{Y}_s, h, l)
            \<\mathbf{X}_{st} \otimes f, h \otimes l\>
  \end{align*}
  Thus, by applying Lemma~\ref{lem:c'-path}, it follows that
  \begin{align*}
    R^f_{st} & = \delta (\partial^f_X Y)_{st} - \sum_{h, l} \sum_g c_r(g, h, l) \partial^g_X Y_s 
          \<\mathbf{X}_{st}, h\> \<f, l\>\\
    & = \delta (\partial^f_X Y)_{st} - \sum_{g, h} c_r(g, h, f) \partial^g_X Y_s 
          \<\mathbf{X}_{st}, h\> 
  \end{align*}
  as desired.
\end{proof}

\begin{proposition}\label{prop:rough-integral}
  Let \(\mathbf{Y} : [0, T] \to \CH^*_{N - 1}\) be a controlled rough path against the \(\gamma\)-H\"older 
  rough path \(\mathbf{X}\). Then, denoting the two parameter process
  \[I_{st}^a = \sum_{f \in \CF_{N - 1}} \partial^f_X Y_s \< \mathbf{X}_{st}, [f]_a\>\] 
  for any \(a \in I\), we have that 
  \[\|\delta I\|_{(N + 1)\gamma} \lesssim \sum_{f \in \CF_{N - 1}} \|R^f\|_{(N - |f|)\gamma} 
    \|\<\mathbf{X}, [f]\>\|_{(|f| + 1)\gamma}.\]
\end{proposition}

As a consequence of the above proposition, in the case where \((N + 1) \gamma > 1\), we may apply the sewing 
lemma to \(I_{st}\) to define the rough integral 
\[\int_s^t Y_r \dd \mathbf{X}_r = \left(\int_s^t Y_r \dd \mathbf{X}_r^a\right)_{a = 1}^d := 
  \left(\lim_{|\CP| \to 0} \sum_{(u, v) \in \CP} I_{uv}^a\right)_{a = 1}^d\]
with \(\CP\) being a partition of \([s, t]\) and we have the remainder inequality 
\begin{align*}
  & \left\|\int_s^t Y_r \dd \mathbf{X}_r - \left(\sum_{f \in \CF_{N - 1}} \partial^f_X Y_s \< \mathbf{X}_{st}, [f]_a\>\right)_{a = 1}^d\right\|\\
  & \lesssim \left(\sum_{f \in \CF_{N - 1}} \|R^f\|_{(N - |f|)\gamma} \|\mathbf{X}^{[f]}\|_{(|f| + 1)\gamma}\right)|t - s|^{(N + 1)\gamma}
\end{align*}
We note that at this point \(\int_s^{\cdot} Y_r \dd \mathbf{X}_r\) takes value in \(\mathbb{R}^d\) and is notably not a tree/forest 
path. As we would like to talk about rough differential equations, namely we would want to phrase a fixed point 
argument involving rough integrals, we need to extend the \(\mathbb{R}^d\)-valued path \(\int_s^{\cdot} Y_r \dd \mathbf{X}_r\) 
to a forest path \(\int_s^{\cdot} \mathbf{Y}_r \dd \mathbf{X}_r\). This is achieved by specifying 
\(\partial^f_X \left(\int_s^{\cdot} \mathbf{Y}_r \dd \mathbf{X}_r\right) = \< f, \int_s^{\cdot} \mathbf{Y}_r \dd \mathbf{X}_r \>\) 
for each \(f \in \CF_{N - 1}^*\) with 
\begin{enumerate}
  \item \(\partial^{\mathbf{1}}_X \int_s^{\cdot} \mathbf{Y}_r \dd \mathbf{X}_r = \int_s^{\cdot} Y_r \dd \mathbf{X}_r\).
  \item\label{rule2} For \(f \in \CF_{\le N - 2}\), we set \(\partial^{[f]}_X \int_0^{\cdot} \mathbf{Y}_r \dd \mathbf{X}_r 
    = \partial^f_X \mathbf{Y}_{\cdot}\).
  \item For \(f_1 \cdot f_2 \in \CF_{\le N - 1}\), \(\partial^{f_1 \cdot f_2}_X \int_0^{\cdot} \mathbf{Y}_r \dd \mathbf{X}_r = 0\).
\end{enumerate}
Thus, with the context of the construction~\ref{rule2} and rough differential equations in mind, we 
observe that if \(\mathbf{Y}\) is a solution to the rough differential equation
\begin{equation}\label{eq:bRDE}
  \dd \mathbf{Y}_t = b(\mathbf{Y}_t) \dd t + \phi(\mathbf{Y}_t) \dd \mathbf{X}_t,
\end{equation}
\(\mathbf{Y}\) necessarily satisfies the algebraic condition that 
\(\partial^{[f]}_X Y = \partial^f_X \phi(\mathbf{Y})\). This motivates the following definition.

\begin{definition}[Coherence]
  A \(\mathbf{X}\)-controlled rough path \(\mathbf{Y}\) is said to be \(\phi\)-coherent if for all 
  \(f \in \CF_{N - 2}\),
  \[\partial^{[f]}_X Y = \partial^f_X \phi(\mathbf{Y}).\]
\end{definition}

\section{A priori estimates for coherent controlled rough paths}

Throughout this section, we fix \(\mathbf{X}\) to be a level \(N\) branched rough path on \(\mathbb{R}^d\) 
and take \(\mathbf{Y}\) to be a \(\phi\)-coherent \(\mathbf{X}\)-controlled rough path on \(\mathbb{R}^m\) 
for some function \(\phi : \mathbb{R}^m \to \CL(\mathbb{R}^d, \mathbb{R}^m)\).

\subsection{Composition with smooth functions}
It is well known that (cf. \cite[Section 7.3]{Friz:20}), in the case where \(N = 2\), the
composition of controlled rough paths with a smooth function is again a controlled rough path. More precisely, for
\((X, \mathbb{X}) \in \CC^{\gamma}([0, T], V)\), \((Y, Y') \in \CD_X^{2\gamma}([0, T], W)\) and
\(\phi \in C^2(W, \bar W)\), we have that
\[(\phi(Y), \phi(Y)') = (\phi(Y), D\phi(Y) Y') \in \CD_X^{2\gamma}([0, T], \bar W).\]
This fact can be generalized for branched rough paths resulting in the following.

\begin{proposition}[Lemma 8.6 of \cite{Gubinelli:06}]\label{prop:comp} 
  For any \(\phi \in C^{N}\), defining \(\phi(\mathbf{Y})\) by
  \[\phi(\mathbf{Y}) = \sum_{f \in \CF_{N - 1}} \partial^f_X \phi(\mathbf{Y}) f\]
  where
  \begin{equation}\label{eq:deriv-comp-def}
    \partial^f_X \phi(\mathbf{Y}) = \sum_{\tau_1 \cdots \tau_{\# f} = f} \frac{1}{(\# f)!}D^{\# f}
    \phi(Y)[(\partial^{\tau_i}_X Y)_{i = 1}^{\# f}]
  \end{equation}
  for any \(f \in \CF_{N - 1}\), we have that \(\phi(\mathbf{Y})\) is
  also a controlled rough path against \(\mathbf{X}\).
\end{proposition}

As we had assumed that \(\mathbf{Y}\) is \(\phi\)-coherent, Equation~\eqref{eq:deriv-comp-def}
can be iterated as in the following example so that \(\partial^f_X \phi(\mathbf{Y})\) can be
written explicitly as a function \(\partial^f_X \phi\) of \(Y\).
\begin{example}
  Let us compute \(\partial^{\tree<2'>}_X \phi(\mathbf{Y})\) in the case where \(d = 1\):
  \begin{align*}
    \partial^{\tree<2'>}_X \phi(\mathbf{Y})
    & = D \phi(Y)\partial^{\tree<2'>}_X Y = D \phi(Y)\partial^{[(\tree<X'>)^2]}_X Y
      = D \phi(Y)\partial^{(\tree<X'>)^2}_X \phi(\mathbf{Y})\\
    & = D \phi(Y) \frac{1}{2} D^2 \phi(Y)[\partial^{\tree<X'>}_X Y \otimes \partial^{\tree<X'>}_X Y]
      = \frac{1}{2}D \phi(Y)D^2 \phi(Y)[\phi(Y) \otimes \phi(Y)].
  \end{align*}
\end{example}

We introduce the following combinatorial factor \(\Pi : \CF \to \mathbb{N}\) where \(\Pi(\mathbf{1}) = 1\),
and for any \(f \in \CF\)
\[\Pi([f]) = \Pi(f) = \prod_{\tau \in \CT} f_\tau ! \Pi(\tau)^{f_\tau}.\]
\(\Pi(f)\) is known as the symmetry factor of \(f\). Geometrically, for \(\tau \in \CT\), \(\Pi(\tau)\)
is the number of permutations one can do on the branches starting from any
vertices of \(\tau\) without changing the tree. E.g. if \(\tau = \tree<32>\), then \(\Pi(\tau) = 3! 2! = 12\).

The symmetric factor allows us to easily compute the coefficient on the right hand
side of \eqref{eq:deriv-comp-def}. In particular, by noticing that the multinomial coefficient can be
written as
\[\mul(f) = \sum_{\tau_1, \dots, \tau_{\# f} \in \CT_{N - 1}} \mathbf{1}_{\{\tau_1 \cdots \tau_{\# f} = f\}}
  = \binom{\# f}{f_h : h \in f} = \frac{(\# f)!}{\Pi(f)} \prod_{\tau \in f} \Pi(\tau),\]
it follows that
\begin{equation}\label{eq:deriv-comp-alt}
  \partial^f_X \phi(\mathbf{Y}) = \prod_{\tau \in f} \frac{\Pi(\tau)}{\Pi(f)} D^{\# f}\phi(Y)[(\partial^\tau_X Y)_{\tau \in f}].
\end{equation}

In the level \(N = 2\) case, we recall that we have the Taylor estimate
\begin{equation}
  \begin{split}
    \|R^{\phi(\mathbf{Y})}_{st}\|
    \le & \sup_{\lambda \in [0, 1]} \frac{1}{2}\|D^2\phi ((1 - \lambda) Y_s + \lambda Y_t)\| \|Y_{st}\|^2
          + \|D\phi(Y_s) R^{\mathbf{Y}}_{st}\|,
  \end{split}
\end{equation}
The main observation of this section is that this Taylor estimate can be generalized for arbitrary
\(N\) for which we have that:
\begin{equs}
  \|R^{\phi(\mathbf{Y}), f}_{st}\| & \le \frac{\|\delta Y_{st}\|^{N - |f|}}{(N - |f|)!}
  \sup_{\lambda \in [0, 1]}\|D^{N - |f|} \partial^f_X \phi(\lambda \mathbf{Y}_s + (1 - \lambda) \mathbf{Y}_t)\|\\
  & + \text{terms of regularity } (N - |f|)\gamma,
\end{equs}
for any \(f \in \CF_{N - 1}\).
In particular, this is described by Corollary~\ref{cor:remain-phi-Y}. This observation is crucial as
it will allow us to estimate the remainder terms of a composition of any order by the remainder term
of the original controlled rough path on the first order (cf. Lemma~\ref{lem:remain-phi-Y}).

\begin{example}
  Let us compute \(R^{\phi(\mathbf{Y}), \tree<X'>}\) in the case where \(N = 3\) and \(d = 1\).

  Firstly, we compute \(c_r(g, h, \tree<X'>)\) for any \(g, h \in \CF_2\). Since \(N = 2\) and
  \(c_r(g, h, \tree<X'>)\) does not vanish only if \(|g| = |h| + 1\) and \(h \neq \mathbf{1}\),
  we only need to compute the \(c_r(g, h, \tree<X'>)\) for \(h = \tree<X'>\) and \(g \in \{(\tree<X'>)^2, \tree<1'>\}\).
  Thus, computing the reduced coproducts,
  \begin{align*}
    \Delta_r (\tree<X'>)^2 & = 2 \tree<X'> \otimes \tree<X'> + (\tree<X'>)^2 \otimes \mathbf{1}\\
    \Delta_r \tree<1'> & = \tree<X'> \otimes \tree<X'> + \tree<1'> \otimes \mathbf{1},
  \end{align*}
  we have that \(c_r((\tree<X'>)^2, \tree<X'>, \tree<X'>) = 2\) and \(c_r(\tree<1'>, \tree<X'>, \tree<X'>) = 1\).
  Consequently, by Proposition~\ref{prop:remainder} and Equation~\eqref{eq:deriv-comp-def}, we have that
  \begin{align*}
    R^{\phi(\mathbf{Y}), \tree<X'>}_{st}
    & = \delta(\partial^{\tree<X'>}_X \phi(\mathbf{Y}))_{st}
      - 2 \partial^{(\tree<X'>)^2}_X \phi(\mathbf{Y}_s) \<\mathbf{X}_{st}, \tree<X'>\>
      - \partial^{\tree<1'>}_X \phi(\mathbf{Y}_s) \<\mathbf{X}_{st}, \tree<X'>\>\\
    & = \delta(\partial^{\tree<X'>}_X \phi(\mathbf{Y}))_{st}
      - D^2 \phi(Y_s)[\partial^{\tree<X'>}_X Y_s \otimes \partial^{\tree<X'>}_X Y_s \<\mathbf{X}_{st}, \tree<X'>\>]
      - (D \phi(Y_s))^2 \partial^{\tree<X'>}_X Y_s \<\mathbf{X}_{st}, \tree<X'>\>.
  \end{align*}
  On the other hand, by computing
  \begin{align*}
    R^{\mathbf{Y}, \mathbf{1}}_{st}
    & = \delta Y_{st} - \partial^{\tree<X'>}_X Y_s \<\mathbf{X}_{st}, \tree<X'>\>
      - \partial^{\tree<1'>}_X Y_s \<\mathbf{X}_{st}, \tree<1'>\>
  \end{align*}
  we have that \(\partial^{\tree<X'>}_X Y_s \<\mathbf{X}_{st}, \tree<X'>\> = \delta(Y)_{st} + o(2\gamma)\)
  where we have denoted \(o(2\gamma)\) for terms of regularity \(2\gamma\). Thus,
  \begin{align*}
    R^{\phi(\mathbf{Y}), \tree<X'>}_{st} & = \delta(D\phi(Y)\phi(Y))_{st}
      - D^2 \phi(Y_s) \phi(Y_s) \delta Y_{st} - (D\phi(Y_s))^2 \delta Y_{st} + o(2\gamma)\\
      & = \delta(D\phi(Y)\phi(Y))_{st} - D(D\phi(Y)\phi(Y))\delta Y_{st} + o(2\gamma)
  \end{align*}
  from which we observe the above estimate by first order Taylor expansion.
\end{example}

For a general level \(N\), this observation is straightforward in the special case where \(\tau = \mathbf{1}\):
\begin{proposition}
  Taking \(\mathbf{X}, \mathbf{Y}\) as above, we have
  \[R^{\phi(\mathbf{Y}), \mathbf{1}}_{st} =
    \delta(\phi(Y))_{st} - \sum_{k = 1}^{N - 1} \frac{1}{k!} D^k \phi(Y_s)[(\delta Y_{st})^{\otimes k}] - R^\sharp_{st}\]
  where
  \[R^\sharp_{st} := \sum_{k = 1}^{N - 1} \sum_{m = 1}^k \frac{(-1)^{m}}{k!} \binom{k}{m}
    D^k \phi(Y_s)\left[(R^{\mathbf{Y}, \mathbf{1}}_{st})^{\otimes m} \otimes (\delta Y_{st})^{\otimes (k - m)}\right],\]
  which has H\"older regularity \(N\gamma\).
\end{proposition}
\begin{proof}
  Observing, for any \(f, g \in \CF_{N - 1} \setminus \{\mathbf{1}\} =: \CF_{N - 1}^+\), we have that
  \begin{equation}\label{eq:crone}
    c_r(f, g, \mathbf{1}) = \<f, g\>.
  \end{equation}
  Thus, by Proposition~\ref{prop:remainder} and the fact that \(\partial^f_X Y = 0\)
  for all \(f \in \CF_{N - 1}^+ \setminus \CT_{N - 1}\),
  \begin{equation}\label{eq:remain-Y-one}
    \delta Y_{st} - R^{\mathbf{Y}, \mathbf{1}}_{st}
    = \sum_{\tau \in \CT_{N - 1}^+} \partial^\tau_X Y_s \<\mathbf{X}_{st}, \tau\>.
  \end{equation}
  Hence, (setting \(\partial^f_X Y = 0\) for all \(f \in \CF\) with \(|f| \ge N\)) we observe that
  \begin{align*}
    \sum_{f \in \CF^+_{N - 1}}\partial^f_X \phi(\mathbf{Y}_s) \<\mathbf{X}_{st}, f\>
    = & \sum_{f \in \CF^+_{N - 1}}\sum_{\tau_1 \cdots \tau_{\# f} = f} \frac{1}{(\# f)!} D^{\# f}
      \phi(Y_s)\left[\bigotimes^{\# f}_{i = 1} \partial^{\tau_i}_X Y_s\right] \<\mathbf{X}_{st}, f\>\\
    = & \sum_{k = 1}^{N - 1} \frac{1}{k!} D^k \phi(Y_s)\left[\left(\sum_{\tau \in \CT^+_{N - 1}}
      \partial^{\tau}_X Y_s \<\mathbf{X}_{st}, \tau\> \right)^{\otimes k}\right]\\
    = & \sum_{k = 1}^{N - 1} \frac{1}{k!} D^k \phi(Y_s)\left[(\delta Y_{st} - R^{\mathbf{Y}, \mathbf{1}}_{st})^{\otimes k}\right].
  \end{align*}
  where the first equality is due to Proposition~\ref{prop:comp}. Thus,
  \begin{equs}
    R^{\phi(\mathbf{Y}), \mathbf{1}}_{st}
      & = \delta(\phi(Y))_{st} - \sum_{f \in \CF^+_{N - 1}}\partial^f_X \phi(\mathbf{Y}_s) \<\mathbf{X}_{st}, f\>\\
      & = \delta(\phi(Y))_{st} - \sum_{k = 1}^{N - 1} \frac{1}{k!} D^k \phi(Y_s)\left[(\delta Y_{st})^{\otimes k}\right] - R^\sharp_{st}
  \end{equs}
  with \(R^\sharp\) defined as in the statement of the proposition.
\end{proof}

The higher order remainder terms (i.e. \(R^{\phi(\mathbf{Y}), \tau}_{st}\) for \(\tau \in \CF_{N - 2}\)),
satisfy a similar relation. We state a reformulation of \cite[Corollary 3.9]{Bonnefoi:22}.
\begin{lemma}
  For any \(f \in \CF_{N - 1}\), we have that
  \[R^{\phi(\mathbf{Y}), f}_{st} = \delta(\partial^f_X \phi(\mathbf{Y}))_{st} -
    \sum_{k = 1}^{N - |f| - 1} \frac{1}{k!}D^k[\partial^f_X \phi(\mathbf{Y}_s)]
    \left[\left(\sum_{1 \le |\tau| \le N - k}\partial^\tau_X Y_s \<\mathbf{X}_{st}, \tau\>\right)^{\otimes k}\right].\]
\end{lemma}

Thus, as a direct consequence of Equation~\eqref{eq:remain-Y-one}, we have the following Taylor
expansion for higher order remainders.
\begin{corollary}\label{cor:remain-phi-Y}
  For any \(f \in \CF_{N - 1}\), we have that
  \[R^{\phi(\mathbf{Y}), f}_{st} = \delta(\partial^f_X \phi(\mathbf{Y}))_{st} -
    \sum_{k = 1}^{N - |f| - 1} \frac{1}{k!}D^k[\partial^f_X \phi(\mathbf{Y}_s)]
    \left[(\delta Y_{st})^{\otimes k}\right]
    - R^{\sharp, f}_{st}\]
  where \(R^{\sharp, f}_{st}\) has regularity \(\ge N\gamma\).
  More precisely,
  \begin{equation}\label{eq:R-sharp-def}
    R^{\sharp, f}_{st} =
    \sum_{k = 1}^{N - |f| - 1} \sum_{\substack{m_1 + m_2 + m_3 = k\\ m_1 \neq k}}
    \frac{(-1)^{m_2 + m_3}}{k!}\binom{k}{m_1, m_2, m_3} D^k[\partial^f_X \phi(\mathbf{Y}_s)][A_{st}^{m_1, m_2, m_3}],
  \end{equation}
  for which the \(m_1, m_2\) and \(m_3\) are summed over the non-negative integers and
  \begin{equation}\label{eq:A-def}
    A_{st}^{m_1, m_2, m_3} :=
    (\delta Y_{st})^{\otimes m_1} \otimes (R^{\mathbf{Y}, \mathbf{1}}_{st})^{\otimes m_2} \otimes
    \left(\sum_{\tau \in \CT_{N - 1} \setminus \CT_{N - k}} \partial^\tau_X Y_s \<\mathbf{X}_{st}, \tau\>\right)^{\otimes m_3}.
  \end{equation}
\end{corollary}

\begin{lemma}\label{lem:remain-phi-Y}
  For any \(f \in \CF_{N - 1}\), we have the estimate
  \begin{equation}\label{eq:remain-phi-Y}
    \begin{split}
      &\ \|R^{\phi(\mathbf{Y}), f}\|_{(N - |f|)\gamma}\\
      \lesssim\ & \|R^{\sharp, f}\|_{N\gamma} |t - s|^{|f|\gamma}
      + \frac{|\CT_{N - 1}|^{N - |f| - 1}}{(N - |f|)!} C_{[s, t]}^f
        \|R^{\mathbf{Y}, \mathbf{1}}\|_{N\gamma}^{N - |f|} |t - s|^{(N - 1)(N - |f|)\gamma}\\
      & + \frac{|\CT_{N - 1}|^{N - |f| - 1}}{(N - |f|)!} C_{[s, t]}^f \sum_{\tau \in \CT^+_{N - 1}}
        (\|\partial^\tau_X Y\|_{\infty; [s, t]}\|\<\mathbf{X}, \tau\>\|_{|\tau|\gamma} |t - s|^{(|\tau| - 1)\gamma})^{(N - |f|)}.
    \end{split}
  \end{equation}
  where for any \(I \subseteq \mathbb{R}\), we have denoted
  \begin{equation}\label{eq:Cf-def}
    C_{I}^f := \sup_{\substack{u < v \in I\\ \lambda \in [0, 1]}}
    \|D^{N - |f|} [\partial^f_X \phi(\lambda \mathbf{Y}_u + (1 - \lambda) \mathbf{Y}_v)]\|.
  \end{equation}
\end{lemma}
\begin{proof}
  Applying the Taylor theorem to Corollary~\ref{cor:remain-phi-Y}, we have that
  \begin{align*}
    \|R^{\phi(\mathbf{Y}), f}\|_{(N - |f|)\gamma} \le &
    \frac{C_{[s, t]}^f}{(N - |f|)!} \|Y\|_{\gamma}^{N - |f|} + \|R^{\sharp, f}\|_{N\gamma} |t - s|^{|f|\gamma}
  \end{align*}
  Thus, we obtain the desired inequality as
  \[\|Y\|_{\gamma} \le \|R^{\mathbf{Y}, \mathbf{1}}\|_{N\gamma} |t - s|^{(N - 1)\gamma} + \sum_{\tau \in \CT^+_{N - 1}}
    \|\partial^\tau_X Y\|_{\infty; [s, t]}\|\<\mathbf{X}, \tau\>\|_{|\tau|\gamma} |t - s|^{(|\tau| - 1)\gamma}\]
  which follows by Equation~\eqref{eq:remain-Y-one}.
\end{proof}

\subsection{An a priori estimate of the derivative}\label{sec:apriori}
As we will need to control the coefficient arising from the Taylor estimate of Lemma~\ref{lem:remain-phi-Y}, 
we will in this section establish a useful bound on the derivatives \(\partial^f_X \phi(\mathbf{Y})\) 
in terms of estimates on \(\phi\) and its derivatives.

\begin{definition}
  For \(f \in \CF_N\) and \(k \in \mathbb{N} \cup \{0\}\), we define \(\Lambda_k(f)\) inductively by 
  \(\Lambda_k(\mathbf{1}) = 0\), and 
  \[\Lambda_k([f]) = \mathbf{1}_{\# f = k} + \Lambda_k(f) = \mathbf{1}_{\# f = k} + \sum_{\tau \in f} \Lambda_k(\tau).\]
  Namely, for any \(k\), \(\Lambda_k(f)\) counts the number of vertices in \(f\) with exactly \(k\) children.
\end{definition}

\begin{lemma}\label{lem:sum-leaves}
  For all \(f \in \CF_{N}\) and constant \(\beta\), we have that
  \begin{equation}\label{eq:sum-Lambda-eq}
    \sum_{k = 0}^{N - 1}\Lambda_k(f)(1 - \beta k) = (1 - \beta)|f| + \beta \# f.
  \end{equation}
\end{lemma}
\begin{proof}
  It suffices to show \(\sum_{k = 0}^{N - 1}\Lambda_k(f)(1 - k) = \# f\) since if this were the case,
  \[\sum_{k = 0}^{N - 1}\Lambda_k(f)(1 - \beta k) = 
    (1 - \beta)\sum_{k = 0}^{N - 1} \Lambda_k(f) + \sum_{k = 0}^{N - 1}\Lambda_k(f)(1 - k) 
    = (1 - \beta)|f| + \beta \# f\]
  as desired. 
  
  For this, we prove by induction on \(f \in \CT_{N}\). The equality is
  clear for the base case \(f = \mathbf{1}\). Thus, taking \(f = \prod_{i = 1}^m [h_i]\) and assuming \eqref{eq:sum-Lambda-eq}
  for each \([h_i]\), it remains to show \eqref{eq:sum-Lambda-eq} for \([f]\). Indeed, this is the case as 
  \begin{align*}
    \sum_{k = 0}^{N - 1} \Lambda_k([f])(1 - k)
    & = \sum_{k = 0}^{N - 1} \left(\mathbf{1}_{k = m} + \sum_{i = 1}^m \Lambda_k([h_i])\right)(1 - k)\\
    & = \sum_{k = 0}^{N - 1} \mathbf{1}_{k = m}(1 - k) 
      + \sum_{k = 0}^{N - 1}\sum_{i = 1}^m \Lambda_k([h_i])(1 - k) = (1 - m) + m = 1.
  \end{align*}
  Thus, for a general forest \(f\), we have 
  \[\sum_{k = 0}^{N - 1} \Lambda_k(f)(1 - k) = \sum_{k = 0}^{N - 1} \sum_{\tau \in f}\Lambda_k(\tau)(1 - k) = \# f\]
  as claimed. 
\end{proof}

We remark that, rather than the induction argument as presented above, the equality
\(\sum_{k = 0}^{N - 1}\Lambda_k(f)(1 - k) = \# f\) alternatively follows by Euler's formula for planar graphs,
namely, \(|\tau| - |E(\tau)| = 1\) for any tree \(\tau\) with \(E(\tau)\) being the set of edges in \(\tau\).

\begin{lemma}\label{cor:partial-bound2}
  Assuming \(\phi\) is such that \(\|D^k\phi(x)\| \le l(\|x\|)^{\alpha_k}\) for all \(k = 0, 1, \dots, N\), 
  we have for any \(f \in \CF_{N - 1}\), 
  \begin{equation}\label{eq:partial-bound2}
    \|\partial^f_X \phi(\mathbf{Y}_t)\| = \|\partial^{[f]}_X Y_t\|
      \le \frac{1}{\Pi(f)} l(\|Y_t\|)^{\sum_{k = 0}^{N - 1} \Lambda_k([f]) \alpha_k}.
  \end{equation}
  Moreover, if \(\alpha_k = (1 - \beta k)\alpha\) for some constants \(\alpha\) and \(\beta\), then 
  \begin{equation}\label{eq:pb2-deriv}
    \|D^n \partial^f_X \phi(\mathbf{Y}_t)\| \le \frac{(1 + |f|)^n}{\Pi(f)} 
      l(\|Y_t\|)^{((1 + |f|) - (n + |f|)\beta)\alpha}
  \end{equation}
  for any \(n \in \mathbb{N} \cup \{0\}\), \(n \le N - |f|\). 
\end{lemma}
\begin{proof}
  We will first show \eqref{eq:partial-bound2} by inducting on \(f\).
  Clearly, the estimate holds for 
  \(f = \mathbf{1}\) and so, taking \(f = \prod_{i = 1}^k [h_i] \in \CF^+_{N - 1}\), it suffices to show that 
  \eqref{eq:partial-bound2} holds for \(f\) assuming that it holds for each \(h_i\). Indeed, we compute
  \begin{align*}
    \|\partial^f_X \phi(\mathbf{Y}_t)\|
    & = \left(\prod_{i = 1}^k\frac{\Pi(h_i)}{\Pi(f)}\right) \|D^k \phi(Y_t)[(\partial^{h_i}_X \phi(\mathbf{Y}_t))_{i = 1}^k]\|\\ 
    & \le \left(\prod_{i = 1}^k\frac{\Pi(h_i)}{\Pi(f)}\right) l(\|Y_t\|)^{\alpha_k} 
      \prod_{i = 1}^k \frac{1}{\Pi(h_i)}l(\|Y_t\|)^{\sum_{j = 0}^{N - 1} \Lambda_j([h_i]) \alpha_j} \\
    & = \frac{1}{\Pi(f)} l(\|Y_t\|)^{\sum_{j = 0}^{N - 1} \Lambda_j([f]) \alpha_j}
  \end{align*}
  as required.

  Now, assuming \(\alpha_k = (1 - \beta k)\alpha\), we show 
  \begin{equation}
    \|D^n \partial^f_X \phi(\mathbf{Y}_t)\| \le \frac{(1 + |f|)^n}{\Pi(f)} 
      l(\|Y_t\|)^{-\alpha\beta n + \sum_{j = 0}^{N - 1} \Lambda_j([f]) \alpha_j}
  \end{equation}
  from which we obtain the desired estimate by Lemma~\ref{lem:sum-leaves}.
  For this, we induct on \(n\). 
  The base case \(n = 0\) is precisely 
  \eqref{eq:partial-bound2} and thus, it suffices to consider \(n \ge 1\) assuming the relevant inductive 
  hypothesis. Again, the estimate is trivial for \(f = \mathbf{1}\) and so, we may take 
  \(f = \prod_{i = 1}^k [h_i] \in \CF^+_{N - 1}\). We observe 
  \begin{align*}
    & \|D^n \partial^f_X \phi(\mathbf{Y}_t)\|
      = \left(\prod_{i = 1}^k\frac{\Pi(h_i)}{\Pi(f)}\right) \left\|D^n\left[D^k\phi(Y_t)[(\partial^{h_i}_X\phi(\mathbf{Y}_t))_{i = 1}^k]\right]\right\| \\
    \le\ & \left(\prod_{i = 1}^k\frac{\Pi(h_i)}{\Pi(f)}\right) \sum_{m_0 + \cdots + m_k = n} \binom{n}{m_0, \dots, m_k}
      \|D^{k + m_0} \phi(Y_t)\| \prod_{i = 1}^k \|D^{m_i}\partial^{h_i}_X \phi(\mathbf{Y}_t)\| \\
    \le\ & \left(\prod_{i = 1}^k\frac{\Pi(h_i)}{\Pi(f)}\right) \sum_{m_0 + \cdots + m_k = n} \binom{n}{m_0, \dots, m_k}
        l(\|Y_t\|)^{(1 - \beta(k + m_0))\alpha} \\
        & \hspace{4cm} 
        \prod_{i = 1}^k \frac{(1 + |h_i|)^{m_i}}{\Pi(h_i)} l(\|Y_t\|)^{-\alpha\beta m_i + \sum_{j = 0}^{N - 1} \Lambda_j([h_i]) \alpha_j} \\
    =\ & \frac{1}{\Pi(f)} \sum_{m_0 + \cdots + m_k = n} \binom{n}{m_0, \dots, m_k} \prod_{i = 1}^k (1 + |h_i|)^{m_i}\\
      & \hspace{4cm} l(\|Y_t\|)^{\sum_{j = 0}^{N - 1} 
        \left(\mathbf{1}_{j = k} + \sum_{i = 1}^k \Lambda_j([h_i])\right)\alpha_j - \alpha\beta \sum_{i = 0}^k m_i}\\
    =\ & \frac{1}{\Pi(f)} l(\|Y_t\|)^{-\alpha\beta n + \sum_{j = 0}^{N - 1} \Lambda_j([f]) \alpha_j} 
      \left(1 + \sum_{i = 1}^k (1 + |h_i|)\right)^n
  \end{align*}
  where we had used the inductive hypothesis to arrive at the third line. Thus, we arrive at the 
  desired inequality by observing that 
  \(\sum_{i = 1}^k (1 + |h_i|) = \sum_{i = 1}^k |[h_i]| = |f|\). 
\end{proof}

\section{Non-explosion of branched RDEs}

\subsection{Branched RDEs with drift and unbounded derivatives}\label{sec:non-exp}
From this point forward, we consider branched RDEs with drift of the form 
\begin{equation}
  \dd \mathbf{Y}_t = b(\mathbf{Y}_t) \dd t + \phi(\mathbf{Y}_t) \dd \mathbf{X}_t.
\end{equation}
Here \(b(\mathbf{Y}_t) := b(Y_t)\) and we denote the rough integral 
\(\int_0^t \phi(\mathbf{Y}_s) \dd \mathbf{X}_s =: \mathbf{Z}_t\) which is also a controlled rough path 
against \(\mathbf{X}\). More precisely, an \(\mathbf{X}\)-controlled rough path 
\(\mathbf{Y} : [0, T] \to \CH_{N - 1}\) is said to be a solution of the branched RDE \eqref{eq:bRDE} if
\begin{itemize}
  \item \(\partial^{\mathbf{1}}_X \mathbf{Y}_t = Y_t = Y_0 + \int_0^t b(Y_s) \dd s + \int_0^t \phi(Y_s) \dd \mathbf{X}_s\),
  \item and for all \(f \in \CF^+_{N - 1}\), \(\partial^f_X \mathbf{Y}_t = \partial^f_X \mathbf{Z}_t\).
\end{itemize}
As remarked before, \(\mathbf{Y}\) is automatically \(\phi\)-coherent and moreover, 
\begin{equation}\label{eq:remain-Y-one-Z}
  R^{\mathbf{Y}, \mathbf{1}}_{st} = R^{\mathbf{Z}, \mathbf{1}}_{st} + B^\sharp_{st}
\end{equation}
where we denote \(B^\sharp_{st} := \int_s^t b(\mathbf{Y}_r) \dd r\). Moreover, by the sewing lemma, 
we have the following estimate for \(R^{\mathbf{Z}, \mathbf{1}}\).
\begin{lemma}\label{lem:remain-Z-one}
  Denoting \(\mathbf{Z}\) as above, we have 
  \begin{equation}\label{eq:remain-Z-one}
    \begin{split}
      \|R^{\mathbf{Z}, \mathbf{1}}\|_{N\gamma; [s, t]} 
      \lesssim & \sum_{f \in \CF_{N - 1}}\|\<\mathbf{X}, [f]\>\|_{(|f| + 1)\gamma}|t - s|^{\gamma} 
          \bigg(\|R^{\sharp, f}\|_{N\gamma; [s, t]}
          |t - s|^{|f|\gamma} \\
      &\hspace{1cm} + C_{[s, t]}^f
          \left(\|R^{\mathbf{Z}, \mathbf{1}}\|_{N\gamma; [s, t]}^{N - |f|} + \|B^\sharp\|_{N\gamma; [s, t]}^{N - |f|}\right) 
          |t - s|^{(N - 1)(N - |f|)\gamma} \\
      &\hspace{1cm} + C_{[s, t]}^f 
          \sum_{\tau \in \CT^+_{N - 1}} 
          \left(\|\partial^\tau_X Y\|_{\infty; [s, t]}\|\<\mathbf{X}, \tau\>\|_{|\tau|\gamma} 
            |t - s|^{(|\tau| - 1)\gamma}\right)^{(N - |f|)}\bigg)\\
      & + \sum_{|f| = N - 1} \|\partial^f_X \phi(\mathbf{Y})\|_{\infty; [s, t]} \|\<\mathbf{X}, [f]\>\|_{N\gamma}.
    \end{split}
  \end{equation}
\end{lemma}
\begin{proof}
  Since \(\partial^f_X Z = 0\) for all \(\# f \ge 2\), we have that
  \begin{align*}
    R^{\mathbf{Z}, \mathbf{1}}_{st} & = \int_s^t \phi(\mathbf{Y}_r) \dd \mathbf{X}_r - \sum_{\tau \in \CT_{N - 1}^+} 
      \partial^\tau_X Z_s \<\mathbf{X}_{st}, \tau\>\\
    & = \int_s^t \phi(\mathbf{Y}_r) \dd \mathbf{X}_r - \sum_{f \in \CF_{N - 2}} 
    \partial^f_X \phi(\mathbf{Y}_s) \<\mathbf{X}_{st}, [f]\>.
  \end{align*}
  Consequently, by the estimate from the sewing lemma~\ref{prop:rough-integral}, we obtain
  \begin{align*}
    \|R^{\mathbf{Z}, \mathbf{1}}_{st}\| \le\ & 
    \left\|\int_s^t \phi(\mathbf{Y}_r) \dd \mathbf{X}_r - \sum_{f \in \CF_{N - 1}} 
    \partial^f_X \phi(\mathbf{Y}_s) \<\mathbf{X}_{st}, [f]\>\right\|
    + \left\|\sum_{|f| = N - 1} \partial^f_X \phi(\mathbf{Y}_s) \<\mathbf{X}_{st}, [f]\>\right\|\\
    \lesssim\ & |t - s|^{(N + 1)\gamma}\sum_{f \in \CF_{N - 1}} \|R^{\phi(\mathbf{Y}), f}\|_{(N - |f|)\gamma; [s, t]}
      \|\<\mathbf{X}, [f]\>\|_{(|f| + 1)\gamma} \\
    & + |t - s|^{N\gamma}\sum_{|f| = N - 1} \|\partial^f_X \phi(\mathbf{Y})\|_{\infty; [s, t]} \|\<\mathbf{X}, [f]\>\|_{N\gamma}. 
  \end{align*}
  Thus, dividing both sides by \(|t - s|^{N \gamma}\), we obtain
  \begin{align*}
    \|R^{\mathbf{Z}, \mathbf{1}}\|_{N\gamma; [s, t]} 
      \lesssim\ & |t - s|^{\gamma}\sum_{f \in \CF_{N - 1}} \|R^{\phi(\mathbf{Y}), f}\|_{(N - |f|)\gamma; [s, t]}
      \|\<\mathbf{X}, [f]\>\|_{(|f| + 1)\gamma} \\
    & + \sum_{|f| = N - 1} \|\partial^f_X \phi(\mathbf{Y})\|_{\infty; [s, t]} \|\<\mathbf{X}, [f]\>\|_{N\gamma}.
  \end{align*}
  Finally, by estimating \(\|R^{\phi(\mathbf{Y}), f}\|_{(N - |f|)\gamma; [s, t]}\) by Equation~\eqref{eq:remain-phi-Y},
  we obtain the desired result.
\end{proof}

For the remainder of this section, we impose the following growth assumption on the coefficients of the branched RDE.
\begin{assumption}\label{as:growth}
  Assume that \(b\) and \(\phi\) are such that for all \(x \in \mathbb{R}^d\),
  \begin{enumerate}[label={(R.\arabic*)}]
    \item\label{eq:kappa_b} \(\|b(x)\| \le l(\|x\|)^{1 + \kappa \gamma}\),
    \item\label{eq:kappa_n} \(\|D^n \phi(x)\| \le l(\|x\|)^{(\theta - n \kappa) \gamma}\) for \(n = 0, \dots, N\)
  \end{enumerate}
  for some non-decreasing function \(l : \mathbb{R}_+ \to \mathbb{R}_+\), \(\theta \in [0, 1)\) and \(\kappa \in [0, \frac{1}{N})\).
\end{assumption}
For all intents and purposes, one may think of \(l(r) = c(1 + r)\) for some \(c > 0\). However, as this 
is not necessary for most of the a priori estimates, we keep the more general form of \(l\) within this assumption.
We also note that for non-explosion, we will require that \(b\) satisfies a growth condition of the form 
\(\<b(x), x\> \le c(1 + \|x\|^2)\) for some \(c > 0\). This assumption is unsurprising as it is necessary 
even for the case of ODEs. 

For all \(r > 0\), we recall that we denote \(\tau^r = \inf \{t \ge 0 : \|Y_t\| \ge r\}\). Moreover, 
let us define
\begin{equation}\label{eq:BR}
  B_r = \{t \in [0, T] : \|Y_t\| \le r\}
\end{equation}
and \(I = I^{k, R}_{\le \epsilon}\) for some interval of length less than or equal to \(\epsilon\) and 
contained in \(B_{R(k + 1)}\) for some \(k \in \mathbb{N}, \epsilon, R > 0\) to be specified 
(in particular, we will specialize 
\(I = [\tau^{Rk}, (\tau^{Rk} + \epsilon) \wedge \tau^{R(k + 1)}]\) in the proof of Theorem~\ref{thm:non-explosion}).
As a direct consequence of Lemma~\ref{cor:partial-bound2}, we have the following estimate.

\begin{corollary}\label{cor:partial-bound}
  Suppose that \ref{eq:kappa_n} holds for all \(n = 0, \dots, N\). For all \(n + |f| \le N\), we have that
  \begin{equation}\label{eq:deriv-partial}
    \|D^n [\partial^f_X \phi(\mathbf{Y})]\|_{\infty; B_r} \lesssim l(r)^{((|f| + 1)\theta - (|f| + n)\kappa)\gamma}.
  \end{equation}
  Consequently, \(C_{I}^f \lesssim l(R(k + 1))^{((|f| + 1)\theta - N\kappa)\gamma}\)
  and for all \(\tau \in \CT^+_{N - 1}\), 
  \[\|\partial^\tau_X Y\|_{\infty; I} \lesssim l(R(k + 1))^{(|\tau|\theta - (|\tau| - 1)\kappa)\gamma}.\]
\end{corollary}
\begin{proof}
  Equation~\eqref{eq:deriv-partial} follows directly from Lemma~\ref{cor:partial-bound2} in which we took 
  \(\alpha = \theta \gamma\) and \(\beta = \frac{\kappa}{\theta}\). 
\end{proof}

We now proceed to estimate the term \(\|R^{\sharp, f}\|_{N\gamma; [s, t]}|t - s|^{(|f| + 1)\gamma}\) 
arising in Lemma~\ref{lem:remain-Z-one}.

\begin{lemma}\label{lem:R-sharp-bd}
  Assuming Assumption~\ref{as:growth} and setting 
  \begin{equation}\label{eq:h-def}
    h(\epsilon) := \epsilon^{(N - 1 + \theta)\gamma}\|R^{\mathbf{Z}, \mathbf{1}}\|_{N\gamma; I^{k, R}_{\le \epsilon}},
  \end{equation}
  we have for all \(\epsilon \le l(R(k + 1))^{-1} \wedge 1\),
  \begin{align*}
    & \epsilon^{(N - 1 + \theta)\gamma} \epsilon^{(|f| + 1)\gamma} \|R^{\sharp, f}\|_{N\gamma; I}
    \lesssim \sum_{k = 0}^{N - |f| - 1} 
      \epsilon^{(1 - \theta + \kappa)k \gamma}h(\epsilon)^k.
  \end{align*}
\end{lemma}
\begin{proof}
  We estimate each summand appearing in \(\|R^{\sharp, f}\|_{N\gamma; I} \epsilon^{(|f| + 1)\gamma}\).
  Fixing \(f \in \CF_{N - 1}\), \(k \in \{1, \dots, N - |f| - 1\}\), \(m_1 + m_2 + m_3 = k\) and \(m_1 \neq k\), 
  let us denote \(A^{m_1, m_2, m_3}\) as in Equation~\eqref{eq:A-def}. Then, by Corollary~\ref{cor:partial-bound},
  and Equation~\eqref{eq:remain-Y-one}, we have that
  \begin{equs}\label{eq:R-sharp-bd1}
      & \|D^k[\partial^f_X \phi(Y)]\|_{\infty; I} \|A^{m_1, m_2, m_3}\|_{N\gamma; I}\\[2ex]
      \lesssim\ 
      & 
      \|R^{\mathbf{Y}, \mathbf{1}}\|_{N\gamma; I}^{m_1 + m_2}
      \sum_{|\tau| > N - k} 
      \|\partial^\tau_X Y\|_{\infty; I}^{m_3}
      \epsilon^{(- (|f| + 1)\theta + (|f| + k)\kappa + N(m_1 + m_2 - 1) + |\tau|m_3)\gamma}\\
      +\ & 
      \|R^{\mathbf{Y}, \mathbf{1}}\|_{N\gamma; I}^{m_2}
        \sum_{\substack{|\tau| > N - k\\ \sigma \in \CT_{N - 1}^+}}
          \|\partial_X^\sigma Y\|^{m_1}_{\infty; I}\|\partial_X^\tau Y\|^{m_3}_{\infty; I}
          \epsilon^{(- (|f| + 1)\theta + (|f| + k)\kappa + N (m_2 - 1) + |\sigma| m_1 + |\tau| m_3)\gamma}
  \end{equs}
  in which we have absorbed the terms involving the H\"older norms of \(\mathbf{X}\) into the constant 
  indicated by \(\lesssim\). We estimate each term individually. By Corollary~\ref{cor:partial-bound} 
  and the fact that \(k - m_3 = m_1 + m_2 \), we find
  \begin{align*}
    & \epsilon^{(|f| + 1)\gamma} \sum_{|\tau| > N - k} 
        \|\partial^\tau_X Y\|_{\infty; I}^{m_3}
        \epsilon^{(- (|f| + 1)\theta + (|f| + k)\kappa)\gamma}
        \epsilon^{N(m_1 + m_2 - 1)\gamma}\epsilon^{|\tau|m_3\gamma}\\
    \lesssim\ & \sum_{|\tau| > N - k}  
        \epsilon^{((1 - \theta)(1 + |f| + |\tau|m_3) + \kappa(|f| + |\tau|m_3 + m_1 + m_2) + N(m_1 + m_2 - 1))\gamma}\\
    \lesssim\ & \epsilon^{((1 - \theta)(1 + |f|) + \kappa(|f| + m_1 + m_2) + N(m_1 + m_2 - 1))\gamma}.
  \end{align*}
  Similarly, we have 
  \begin{align*}
    & \epsilon^{(|f| + 1)\gamma} \sum_{\substack{|\tau| > N - k\\ \sigma \in \CT_{N - 1}^+}}
    \|\partial_X^\sigma Y\|^{m_1}_{\infty; I}\|\partial_X^\tau Y\|^{m_3}_{\infty; I}
    \epsilon^{(- (|f| + 1)\theta + (|f| + k)\kappa)\gamma}
    \epsilon^{N(m_2 - 1)\gamma}\epsilon^{|\sigma| m_1\gamma}\epsilon^{|\tau| m_3\gamma}\\
    \lesssim\ & \epsilon^{((1 - \theta)(1 + |f|) + \kappa(|f| + m_2) + N(m_2 - 1))\gamma}.
  \end{align*}
  Now by Equation~\eqref{eq:remain-Y-one}, 
  \begin{equation}\label{eq:R-Y-one-bd}
    \|R^{\mathbf{Y}, \mathbf{1}}\|_{N\gamma; I}^{p} 
    \lesssim \|R^{\mathbf{Z}, \mathbf{1}}\|_{N\gamma; I}^{p} + \|B^\sharp\|_{N\gamma; I}^{p}
  \end{equation}
  for any \(p \ge 1\). Thus, by substituting this into Equation~\eqref{eq:R-sharp-bd1}, we can estimate the terms of 
  \eqref{eq:R-sharp-bd1} involving \(\|B^\sharp\|_{N\gamma; I}\). 
  Since by assumption \(\epsilon \le l(R(k + 1))^{-1}\) and for all \(r \in I\), \(\|b(\mathbf{Y}_r)\| \le l(R(k + 1))^{1 + \kappa \gamma}\),
  we observe that
  \begin{equation}\label{eq:B-sharp-reg}
    \|B^\sharp\|_{N\gamma; I} \le \sup_{u < v \in I} \frac{1}{|u - v|^{N\gamma}} \int_u^v \|b(\mathbf{Y}_r)\| \dd r
       \le \epsilon^{1 - N\gamma} l(R(k + 1))^{1 + \kappa \gamma} \le \epsilon^{-(N + \kappa)\gamma}.
  \end{equation}
  Hence, as
  \begin{align*}
    & \|B^\sharp\|_{N\gamma; I}^{m_1 + m_2}
      \epsilon^{((1 - \theta)(1 + |f|) + \kappa(|f| + m_1 + m_2) + N(m_1 + m_2 - 1))\gamma}
    \lesssim \epsilon^{((1 - \theta)(1 + |f|) + \kappa|f| - N)\gamma},
  \end{align*}
  and 
  \begin{align*}
    & \|B^\sharp\|_{N\gamma; I}^{m_2}\epsilon^{((1 - \theta)(1 + |f|) + \kappa(|f| + m_2) + N(m_2 - 1))\gamma}
    \lesssim \epsilon^{((1 - \theta)(1 + |f|) + \kappa|f| - N)\gamma},
  \end{align*}
  by substituting the above estimates into Equation~\eqref{eq:R-sharp-bd1}, it follows that
  \begin{align*}
    & \epsilon^{((N - 1) + \theta)\gamma} \epsilon^{(|f| + 1)\gamma}\|R^{\sharp, f}\|_{N\gamma; I} \\
    \lesssim\ & \sum_{k = 1}^{N - |f| - 1}\sum_{\substack{m_1 + m_2 + m_3 = k\\ m_1 \neq k}} 
    \bigg(\|R^{\mathbf{Z}, \mathbf{1}}\|_{N\gamma; I}^{m_2}
      \epsilon^{(N + \kappa) m_2\gamma} + \|R^{\mathbf{Z}, \mathbf{1}}\|_{N\gamma; I}^{m_1 + m_2} 
      \epsilon^{(N + \kappa)(m_1 + m_2)\gamma} + 1\bigg)\\
    \le\ & \sum_{k = 1}^{N - |f| - 1}\sum_{\substack{m_1 + m_2 + m_3 = k\\ m_1 \neq k}} 
      (\epsilon^{(1 - \theta + \kappa)(m_1 + m_2) \gamma}h(\epsilon)^{m_1 + m_2} 
      + \epsilon^{(1 - \theta + \kappa)m_2 \gamma}h(\epsilon)^{m_2} + 1)\\
    \lesssim\ & \sum_{k = 0}^{N - |f| - 1} 
      \epsilon^{(1 - \theta + \kappa)k \gamma}h(\epsilon)^k
  \end{align*}
  where in the penultimate line, we used the fact that 
  \[\|R^{\mathbf{Z}, \mathbf{1}}\|_{N\gamma; I}^{q}\epsilon^{N q\gamma} = 
    \epsilon^{(1 - \theta)q \gamma}\|R^{\mathbf{Z}, \mathbf{1}}\|_{N\gamma; I}^{q}\epsilon^{(N - (1 - \theta))q\gamma} 
    = \epsilon^{(1 - \theta)q \gamma}h(\epsilon)^{q}\]
  for any \(q > 0\).
\end{proof}

\begin{lemma}\label{lem:h-poly-bdd}
  Under Assumption~\ref{as:growth} and denoting \(h(\epsilon)\) as 
  in Equation~\eqref{eq:h-def}, we set \(\lambda := (1 - \theta + \kappa)\gamma\). Then, we have that
  \begin{equation}\label{eq:h-poly-bdd}
    \begin{split}
      h(\epsilon)
      \lesssim & 
        \sum_{f \in \CF_{N - 1}} \left(\epsilon^{N\lambda} h(\epsilon)^{N - |f|} +
        \sum_{k = 0}^{N - |f| - 1} 
        \epsilon^{k\lambda}h(\epsilon)^k
        \right)
    \end{split}
  \end{equation}
  for all \(\epsilon \le l(R(k + 1))^{-1} \wedge 1\).
\end{lemma}
\begin{proof}
  By multiplying both sides of the estimate provided by Lemma~\ref{lem:remain-Z-one} by \(\epsilon^{(N - 1 + \theta)\gamma}\),
  we have 
  \begin{equation}\label{eq:h-est}
    \begin{split}
      h(\epsilon)
      \lesssim & \sum_{f \in \CF_{N - 1}} 
          \|R^{\sharp, f}\|_{N\gamma; I}
          \epsilon^{|f|\gamma}\epsilon^{((N - 1) + \theta)\gamma}\epsilon^{\gamma} \\
      & + \sum_{f \in \CF_{N - 1}} C_I^f \|R^{\mathbf{Z}, \mathbf{1}}\|_{N\gamma; I}^{N - |f|} 
          \epsilon^{(N - 1)(N - |f|)\gamma} \epsilon^{((N - 1) + \theta)\gamma}\epsilon^{\gamma}\\
      & + \sum_{f \in \CF_{N - 1}} C_I^f \|B^\sharp\|_{N\gamma; I}^{N - |f|} 
          \epsilon^{(N - 1)(N - |f|)\gamma} \epsilon^{((N - 1) + \theta)\gamma}\epsilon^{\gamma}\\
      & + \sum_{f \in \CF_{N - 1}} C_I^f \sum_{\tau \in \CT^+_{N - 1}} \|\partial^\tau_X Y\|_{\infty; I}^{(N - |f|)}
            \epsilon^{(N - |f|)(|\tau| - 1)\gamma}\epsilon^{((N - 1) + \theta)\gamma}\epsilon^{\gamma}\\
      & + \sum_{|f| = N - 1} \|\partial^f_X \phi(\mathbf{Y})\|_{\infty; I} \epsilon^{((N - 1) + \theta)\gamma}.
    \end{split}
  \end{equation}
  The first sum on the right hand side is estimated by Lemma~\ref{lem:R-sharp-bd} and it remains to 
  estimate the remaining four sums.

  By Corollary~\ref{cor:partial-bound} and Equation~\eqref{eq:B-sharp-reg}, we have that, 
  for any \(f \in \CF_{N - 1}\), 
  \begin{align*}
    & C_I^f \|R^{\mathbf{Z}, \mathbf{1}}\|_{N\gamma; I}^{N - |f|} 
      \epsilon^{(N - 1)(N - |f|)\gamma} \epsilon^{(N - 1 + \theta)\gamma}\epsilon^{\gamma}
    \lesssim \epsilon^{N(1 - \theta + \kappa)\gamma}h(\epsilon)^{N - |f|}
  \end{align*}
  and
  \begin{align*}
    & C_I^f \|B^\sharp\|_{N\gamma; I}^{N - |f|} 
      \epsilon^{(N - 1)(N - |f|)\gamma} \epsilon^{(N - 1 + \theta)\gamma}\epsilon^{\gamma}
    \lesssim \epsilon^{|f|(1 - \theta + \kappa)\gamma}.
  \end{align*}
  Moreover, for any \(\tau \in \CT^+_{N - 1}\),
  \begin{align*}
    & C_I^f \|\partial^\tau_X Y\|_{\infty; I}^{(N - |f|)}
      \epsilon^{(N - |f|)(|\tau| - 1)\gamma}\epsilon^{(N - 1 + \theta)\gamma}\epsilon^{\gamma}
    \lesssim \epsilon^{(|\tau|(N - |f|) + |f|)(1 - \theta + \kappa)\gamma}.
  \end{align*}
  Finally, for \(|f| = N - 1\), we have that
  \[\|\partial^f_X \phi(\mathbf{Y})\|_{\infty; I} \epsilon^{(N - 1 + \theta)\gamma}
    \lesssim \epsilon^{(N - 1)(1 - \theta + \kappa)\gamma}.\]
  Hence, we obtain the desired estimate by substituting the above inequalities into Equation~\eqref{eq:h-est}.
\end{proof}

\begin{corollary}\label{cor:h-tilde-est}
  Under Assumption~\ref{as:growth}, there exists some \(K > 0\) such that for any \(R > 0\) and 
  sufficiently large \(k \in \mathbb{N}\), we have that
  \[l(R(k + 1))^{-(N - 1 + \theta)\gamma}\|R^{\mathbf{Z}, \mathbf{1}}\|_{N\gamma; I^{k, R}_{\le l(R(k + 1))^{-1}}} \le K.\]
\end{corollary}
\begin{proof}
  Taking \(h\) as in Equation~\eqref{eq:h-def} and denoting \(\epsilon_k = l(R(k + 1))^{-1} \wedge 1\), we define 
  \[\tilde h(\epsilon) = h(\epsilon) \mathbf{1}_{\epsilon \le \epsilon_k} + 
      h(\epsilon_k) \mathbf{1}_{\epsilon > \epsilon_k}.\]
  Then, by Lemma~\ref{lem:h-poly-bdd}, we have that \(\tilde h\) satisfies the estimate 
  \[\tilde h(\epsilon)
    \lesssim 
      \sum_{f \in \CF_{N - 1}} \left(\epsilon^{N\lambda} \tilde h(\epsilon)^{N - |f|} +
      \sum_{k = 0}^{N - |f| - 1} 
      \epsilon^{k\lambda} \tilde h(\epsilon)^k\right)\]
  for all \(\epsilon \in (0, 1]\). Thus, as a consequence of \cite[Lemma 4.6]{Li:25} (we provide a 
  slightly more general formulation of this result in Lemma~\ref{lem:unif-bound}), it follows that 
  there exists some universal \(\epsilon^*, K > 0\) (which is notably independent of \(R\) and \(k\)) 
  such that \(\tilde h(\epsilon^*) \le K\). Finally, as \(\epsilon_k \le l(R(k + 1))^{-1} \to 0\) as 
  \(k \to \infty\), we have that \(K \ge \tilde h(\epsilon^*) = h(l(R(k + 1))^{-1})\) for sufficiently 
  large \(k\) as desired.
\end{proof}

\begin{theorem}\label{thm:non-explosion}
  Suppose Assumption~\ref{as:growth} holds for \(l\) satisfying \(\int_a^\infty l(r)^{-1} \dd r = \infty\)
  for some \(a > 0\). Moreover, suppose that the vector field \(b\) satisfies
  \begin{equation}\label{eq:lin-growth}
    \<x, b(x)\> \le \|x\| l(\|x\|) \text{ for all } x \in \mathbb{R}^d.
  \end{equation}
  Then, the branched RDE~\eqref{eq:bRDE} has a unique global solution.
\end{theorem}
\begin{proof}
  Under the assumed regularity conditions, the branched RDE~\eqref{eq:bRDE} has a unique
  maximal solution by a standard Banach fixed point argument and so, it suffices to show non-explosion.
  Let \(K > 0\) be the constant obtained from Corollary~\ref{cor:h-tilde-est} and \(R = R(K) > 0\) 
  be a sufficiently large constant depending on \(K\) (e.g. sufficiently large so that it satisfies 
  Equation (16) in \cite{Li:25}). Then, denoting \(\epsilon_k = l(R(k + 1))^{-1}\), 
  we have from Proposition~\ref{prop:remainder} that for sufficiently large \(k \in \mathbb{N}\) and
  for all \(s < t \in I^{k, R}_{\le \epsilon_k}\) 
  \begin{align*}
    \|\delta Z_{st}\| \le\ & \|R^{\mathbf{Z}, \mathbf{1}}_{st}\| 
      + \sum_{\tau \in \CT^+_{N - 1}} \|\partial_X^\tau Y_s\| \|\<\mathbf{X}, \tau\>\|_{|\tau|\gamma}|t - s|^{|\tau|\gamma}\\
    \lesssim\ & \|R^{\mathbf{Z}, \mathbf{1}}\|_{N\gamma; I} \epsilon_k^{(N - 1 + \theta)\gamma} |t - s|^{(1 - \theta)\gamma}
      + \sum_{\tau \in \CT^+_{N - 1}} \epsilon_k^{-(|\tau|\theta - (|\tau| - 1)\kappa)\gamma} |t - s|^{|\tau|\gamma}\\
    \le\ & \|R^{\mathbf{Z}, \mathbf{1}}\|_{N\gamma; I} \epsilon_k^{(N - 1 + \theta)\gamma}|t - s|^{(1 - \theta)\gamma}
    + \sum_{\tau \in \CT^+_{N - 1}} \epsilon_k^{(|\tau| - 1)(1 - \theta)\gamma + (|\tau| - 1)\kappa\gamma}|t - s|^{(1 - \theta)\gamma}\\
    \le\ & (K + |\CT^+_{N -1}|) |t - s|^{(1 - \theta)\gamma}.
  \end{align*}
  where in the second line, we applied the estimate from Corollary~\ref{cor:partial-bound}. 

  Hence, specializing \(I^{k, R}_{\le \epsilon_k} = [\tau^{Rk}, (\tau^{Rk} + \epsilon_k) \wedge \tau^{R(k + 1)}]\),
  it follows that, for some constant \(C > 0\) and sufficiently large \(k\), we have that 
  \[\sup_{t + \tau^{Rk} \in I^{k, R}_{\le \epsilon_k}} \|Z_{t + \tau^{Rk}} - Z_{\tau^{Rk}}\| 
    \le C t^{(1 - \theta) \gamma}.\]
  Thus, we obtain non-explosion by \cite[Lemma 3.5]{Li:25}.
\end{proof}

\subsection{A maximal inequality for branched RDEs}
We in this section provide a maximal inequality for the solution of a branched RDE with drift and unbounded derivatives.
For this, we use a variant of \cite[Lemma 4.6]{Li:25} but while making the constants 
within it explicit. This is useful should one wish to obtain integrability and tail estimates for the 
solution \((x_t)\) of the RDE~\eqref{eq:bRDE} when \(\mathbf{X}\) is a stochastic rough path (e.g. Gaussian).

\begin{lemma}\label{lem:unif-bound}
  Let \((V, \|\cdot\|)\) be a finite-dimensional normed vector space and \(U\) be a closed connected subset of \(V\).
  Suppose \(a, b \in C(U; \mathbb{R})\) for which
  \(\sup_{x \in U: \|x\| \le R} b(x) < \infty\) for all \(R > 0\) and
  \(c \in C(\mathbb{R}_+, \mathbb{R}_+)\) is monotone and \(c(0) = 0\). Then, denoting 
  \[R^* = \sup \{\|x\| \in \mathbb{R}_+ : x \in U, a(x) \le 0\}\]
  and
  \begin{equation}\label{eq:ineq-set}
    \CA =
      \left\{(X_\epsilon) \in C([0, 1], U) \;\middle|\;
        \begin{aligned}
          & \forall s < t,\ \|X_s\| \le \|X_t\|,\\
          & \forall \epsilon \in (0, 1),\ a(X_\epsilon) \le c(\epsilon) b(X_\epsilon)
        \end{aligned}
      \right\},
  \end{equation}
  for any \(r > 0\), if \(\epsilon^* > 0\) is such that, for all \(\epsilon < \epsilon^*\) 
  \[a(x) - c(\epsilon)b(x) > 0, \qquad \forall\ \|x\| = R^* + r,\] 
  then \(\sup_{X \in \CA} \|X_{\epsilon^*}\| \le R^* + r\).
\end{lemma}
\begin{proof}
  Let $ \partial B_R=\{x \in U: \|x\|=R\}$, and define
  \[R^* = \sup \{\|x\|  : x \in U, a(x) \le 0\}.\]
  Since in the case where $R^*=\infty$, the inequality holds trivially so we may assume $R^*<\infty$. 

  For any $r>0$, we observe that \(\inf_{x \in \partial B_{R^*+r}} a(x) > 0\) since $a > 0$ on the compact set $\partial B_{R^*+r}$.
  Thus, as \(\lim_{\epsilon \to 0} c(\epsilon) = 0\) and moreover, \(b\) is bounded on \(B_{R^* + r}\), 
  there exists some \(\epsilon^* \in (0, 1)\) such 
  that for all \(\epsilon < \epsilon^*\), 
   $$a(x) - c(\epsilon) b(x) > 0, \qquad \forall x \in \partial B_{R^*+r}.$$
  Hence, taking \(X \in \CA\), it suffices to show \(\|X_{\epsilon^*}\| \le R^* + r\).
  Suppose, for contradiction, that \(\|X_{\epsilon^*}\| > R^* + r\). 
  Then, using the same reasoning as above, we can pick \(\epsilon' < \epsilon^*\) such that
  $$a(x) - c(\epsilon') b(x) > 0, \qquad \forall x \in \overline{B_{\|X_{\epsilon^*}\|} \setminus B_{R^* + r}}.$$
  Consequently, as \(X \in \CA\), 
  \[a(X_{\epsilon'}) - c(\epsilon') b(X_{\epsilon'}) \le 0 \hbox{ and }
  \|X_{\epsilon'}\| \le \|X_{\epsilon^*}\|,\]
  it follows that \(\|X_{\epsilon'}\| \le R^* + r\).

  Now, since \(\epsilon \mapsto \|X_\epsilon\|\) is continuous and \(\|X_{\epsilon^*}\| > R^* + r\), 
  there exists some \(\epsilon \in (\epsilon', \epsilon^*)\) such that \(\|X_\epsilon\| = R^* + r\).
  However, by the choice of \(\epsilon^*\), \(a(X_\epsilon) - c(\epsilon) b(X_\epsilon) > 0\) 
  which contradicts the fact that \(X \in \CA\).
\end{proof}

We remark that the requirement \(\sup_{x \in U: \|x\| \le R} b(x) < \infty\) is not redundant.
Indeed, in contrast to the linear case, continuity of \(b\) does not necessarily imply boundedness. 
Consider $c_{00}$ -- the space of sequences which are eventually $0$ equipped with the 
supremum norm. Defining $(x^k_n) = (2 \delta_{k, n})$, i.e. the sequence whose value at the $k$-th position 
is $2$ and $0$ everywhere else, we observe that $b: x\mapsto \sum_{n=1}^\infty x_n^n$ is not 
bounded on \(\{x^k\}_{k \in \mathbb{N}} \subseteq B_2\) despite being continuous.

\begin{theorem}\label{thm:integrability}
  Supposing Assumption~\ref{as:growth} and \(b\) has sub-linear growth
  in its radial direction, i.e. there exists some \(\delta > 0\) such that
  \begin{equation}\label{eq:sub-linear}
    \<x, b(x)\> \lesssim \|x\|(1 + \|x\|)^{1 - \delta},\qquad \forall x \in \mathbb{R}^d.
  \end{equation}
  Then, the unique global solution \((x_t)\) of the branched RDE~\eqref{eq:bRDE} satisfies
  \[\sup_{s \in [0, t]}\|x_s\| \lesssim \|\mathbf{X}\|_{\gamma; [0, t]}^\eta\]
  for any \(t \in [0, T]\) where
  \(\eta := (N^2 - N + 1)\left(\left(\frac{1}{\delta} + \frac{N + 1}{(1 - \delta)\lambda}\right) 
    \vee \left(\frac{1}{\delta} + \frac{1}{(1 - \delta)(1 - \theta)\gamma} - 1\right)\right)\)
  and \(\lambda := (1 - \theta + \kappa)\gamma\). 
\end{theorem}

We remark that we had required the sub-linear growth condition due to our application of \cite[Lemma 3.5]{Li:25}. 
In the case where \(b\) has linear growth, \cite[Lemma 3.5]{Li:25} can only provide an exponential estimate 
for the process of the form \(\|x_t\| \lesssim \exp(C\|\mathbf{X}\|_\gamma^\eta)\) in which \(\eta > 2\) 
(in contrast to a polynomial estimate in the sub-linear case). This estimate is in general not very 
useful as the expectation of the right hand side would not be finite even for Gaussian rough paths.

\begin{proof}
  Making the dependence on \(\|\mathbf{X}\|_{\gamma; [0, T]}\) in Lemma~\ref{lem:h-poly-bdd} explicit, it is easy to 
  check Equation~\eqref{eq:h-poly-bdd} becomes
  \begin{align*}
    h(\epsilon) & \lesssim (\|\mathbf{X}\|_\gamma \vee \|\mathbf{X}\|_\gamma^{N + (N - 1)^2})
      \left(\epsilon^{\lambda}\sum_{f \in \CF_{N - 1}} \left(h(\epsilon)^{N - |f|} +
      \sum_{k = 1}^{N - |f| - 1} h(\epsilon)^k\right) + 1\right)\\
    & \lesssim (\|\mathbf{X}\|_\gamma \vee \|\mathbf{X}\|_\gamma^{N + (N - 1)^2})
      \left(\epsilon^{\lambda} \left(h(\epsilon)^N + h(\epsilon)\right) + 1\right).
  \end{align*}
  Thus, denoting \(\Xi := \|\mathbf{X}\|_\gamma \vee \|\mathbf{X}\|_\gamma^{N + (N - 1)^2}\), in applying
  Lemma~\ref{lem:unif-bound} where we take \(r = 1\), 
  \[a(x) = x - c_0 \Xi,\ b(x) = c_0(x^N + x), 
    \text{ and } c(\epsilon) = \epsilon^\lambda \Xi\]
  for some constant \(c_0 > 0\), we obtain 
  \[h(c_1\Xi^{-\lambda^{-1}}(\Xi^N + \Xi + 1)^{-\lambda^{-1}}) \le c_0 \Xi + 1\]
  with \(c_1 > 0\) a small constant. We remark that the above inequality is precisely the estimate occurring in 
  the proof of Corollary~\ref{cor:h-tilde-est} with \(\epsilon^* := c_1\Xi^{-\lambda^{-1}}(\Xi^N + \Xi + 1)^{-\lambda^{-1}}\)
  and \(K := c_0 \Xi + 1\). Thus, by the same argument as Theorem~\ref{thm:non-explosion}, we obtain 
  \[\sup_{t + \tau^{Rk} \in I^{k, R}_{\le \epsilon}} \|Z_{t + \tau^{Rk}} - Z_{\tau^{Rk}}\| 
    \le (c_0 \Xi + |\CT^+_{N - 1}| + 1) t^{(1 - \theta) \gamma}\]
  whenever \(\epsilon \le \epsilon^* = c_1\Xi^{-\lambda^{-1}}(\Xi^N + \Xi + 1)^{-\lambda^{-1}}\) 
  and \(R = C \Xi\) for some sufficiently large constant \(C > 0\) so that it satisfies Equation (16) of \cite{Li:25}. 
  Consequently, setting 
  \begin{align*}
    k_0 & := \inf \{k \in \mathbb{N} : (Rk)^{\delta - 1} \le \epsilon^* \wedge (c_0 \Xi + 1)^{-\frac{1}{(1 - \theta)\gamma}}\}\\
    & \lesssim (\Xi^{((1 - \delta)\lambda)^{-1}}(\Xi^N + \Xi + 1)^{((1 - \delta)\lambda)^{-1}} \vee 
      \Xi^{-1}(\Xi + 1)^{((1 - \delta)(1 - \theta)\gamma)^{-1}})
  \end{align*}
  we have by \cite[Lemma 3.5]{Li:25} that \(\tau^{Rk} - \tau^{R(k - 1)} \ge (Rk)^{\delta - 1}\) for all
  \(k \ge k_0\). Namely, for any \(M \in \mathbb{N}\), if \(T \le \sum_{k = k_0}^M (Rk)^{\delta - 1}\), then
  \(\sup_{t \in [0, T]} \|x_t\| \le RM\). Hence, as 
  \[\sum_{k = k_0}^M (Rk)^{\delta - 1} \ge R^{\delta - 1}\delta^{-1}(M^\delta - k_0^\delta),\]
  it follows that 
  \[\left\{T \le R^{\delta - 1}\delta^{-1}(M^\delta - k_0^\delta)\right\} 
    \subseteq \left\{\sup_{t \in [0, T]} \|x_t\| \le RM \right\}.\]
  Thus, for any \(r > 0\), by setting \(M = \left\lfloor \frac{r}{R}\right\rfloor\), we have that
  \(\left(\frac{r}{R}\right)^\delta \ge M \ge \left(\frac{r}{R}\right)^\delta - 1\) and so,
  \begin{align*}
    \left\{\sup_{t \in [0, T]} \|x_t\| \le r \right\} 
    & \supseteq \left\{\sup_{t \in [0, T]} \|x_t\| \le R \left\lfloor \frac{r}{R}\right\rfloor \right\}\\
    & \supseteq \left\{T \le R^{\delta - 1}\delta^{-1}\left(\left\lfloor \frac{r}{R}\right\rfloor^\delta - k_0^\delta\right)\right\}\\
    & \supseteq \left\{C'\left(1 \vee \Xi^{\left(\frac{1}{\delta} + \frac{N + 1}{(1 - \delta)\lambda}\right) 
      \vee \left(\frac{1}{\delta} + \frac{1}{(1 - \delta)(1 - \theta)\gamma} - 1\right)}\right) \le r\right\}
  \end{align*}
  as claimed.
\end{proof}

\subsection{Beyond \texorpdfstring{\(\gamma\)}{γ}-growth}\label{sec:beyond}

Recalling the proof of Lemma~\ref{lem:h-poly-bdd}, we observe that the growth 
assumption~\ref{as:growth} was chosen so that the coefficients in the sum of the estimate provided by 
Lemma~\ref{lem:remain-Z-one} are controlled. As demonstrated by Lemma~\ref{cor:partial-bound},
these coefficients depend on the growth rates of \(\|D^n \phi\|\) and we expect that the growth 
assumption can be relaxed if one can control the growth of these coefficients by allowing for faster 
decay of the derivatives of \(\phi\). This turns out to be the case and we in this section show 
that non-explosion remains to hold for \(\phi\) with higher 
growth rate should it be compensated with decaying derivatives (of order \(\ge 2\)). 

Throughout this subsection, we take \(\mathbf{X}\) to be a branched rough path with regularity 
\(\gamma \in (\frac{1}{N + 1}, \frac{1}{N}]\) and \(Y\) be a solution of the RDE~\eqref{eq:rde-strat} 
with the coefficients \(\phi\) now satisfying the following growth assumption.

\begin{assumption}\label{as:growth2}
  Assume that there exists a non-decreasing function \(l : \mathbb{R}_+ \to [1, \infty)\) such that 
  \begin{itemize}
    \item the map \(x \mapsto \frac{l(2x)}{l(x)}\) is monotone (either increasing or decreasing) and
    \item there exists some constant \(c > 0\) such that \(l(x) \le c(1 + x)\) for all \(x \ge 0\),
  \end{itemize}
  and there exists a constant \(\alpha \in [(N - 1)\gamma, N \gamma)\) such that 
  \(\phi \in C^{N + 1}(\mathbb{R}; \CL(\mathbb{R}^d, \mathbb{R}))\) satisfies
  \(\|D\phi(x)\| \le (\log l(|x|))^{\gamma}\) and 
  \begin{equation}
    \|D^n \phi(x)\| \le l(\|x\|)^{(1 - n)\alpha},
  \end{equation}
  for all \(n \in \{0, 2, 3, \dots, N\}\) and \(x \in \mathbb{R}^d\). 
\end{assumption}

We remark that this assumption is phrased using the function \(l\) instead of directly requiring 
\(\|D^n \phi(x)\| \le c(1 + \|x\|)^{(1 - n)\alpha}\) as the latter requires stronger 
decay of the derivatives of \(\phi\). Namely, by allowing the choice of \(l\), one allows a trade 
off between the growth of \(\phi\) and the decay of its derivatives. 

Moreover, under this growth assumption,  
\(\|D\phi(x)\| \le (\log l(|x|))^{\gamma} \lesssim l(|x|)^{\frac{\delta}{N}}\) for any \(\delta > 0\). Thus, 
we can read directly from Equation~\eqref{eq:partial-bound2} and Lemma~\ref{lem:sum-leaves} that 
\begin{equation}\label{eq:partial-bound3}
  \|D^n \partial^f_X \phi(\mathbf{Y}_t)\| \lesssim l(|Y_t|)^{(1 - n)\alpha + \frac{\delta}{N} (\Lambda_0([f]) + \Lambda_1([f]))}
  \le l(|Y_t|)^{(1 - n)\alpha + \delta} 
\end{equation}
where the constant indicated by \(\lesssim\) depends on \(\delta\). However, this estimate turns out to be 
insufficient for our application when \(n = 1\) and we provide a slightly more precise estimate in 
Corollary~\ref{cor:partial-bound3}.

With this a priori estimate in mind, we proceed as in Section~\ref{sec:non-exp} by estimating the remainder 
\(\|R^{\mathbf{Y},\mathbf{1}}\|\). As the computation is similar, we relegate the details of the proof to the
appendix.

For any \(r > 0\), we define 
\begin{equation}\label{eq:tau-sigma-r}
  \tilde\tau^{r} = \inf\{t : |Y_t| \ge 2r\} \text{ and }
  \tilde\sigma^{r} = \sup\{t < \tilde\tau^{r} : |Y_t| \le r\}
\end{equation}
and we set the interval \(I_r = [\tilde\sigma^r, \tilde\tau^{r}]\) so that \(r \le |Y_t| \le 2r\) for any \(t \in I_r\). 
We have the following control on the remainder term on this interval.

\begin{lemma}\label{lem:R-est}
  Assuming assumption~\ref{as:growth2}, for any \(\delta \in (0, \frac{\gamma}{N})\), 
  there exists a constant \(K > 0\) such that for sufficiently large \(r > 0\) and interval \(I \subseteq I_r\) of size 
  \(|I| \le l(2r)^{-1}\), we have 
  \[l(2r)^{- (\alpha + \delta)} \|R^{\mathbf{Y}, \mathbf{1}}\|_{N\gamma; I} \le K.\]
\end{lemma}
\begin{proof}
  This result holds by the same argument as in the proof of Corollary~\ref{cor:h-tilde-est}
  where we replace the estimate from Lemma~\ref{lem:h-poly-bdd} with Lemma~\ref{lem:h-poly-bdd2}.
\end{proof}

\begin{theorem}\label{thm:non-explosion2}  
  Assuming \(\phi \in C^{N + 1}(\mathbb{R}; \CL(\mathbb{R}^d, \mathbb{R})) \simeq C^{N + 1}(\mathbb{R}; \mathbb{R}^d)\)
  satisfies Assumption~\ref{as:growth2}. 
  Then, the branched RDE~\eqref{eq:rde-strat} has a unique global solution.  
\end{theorem}
\begin{proof}
  In the following application, for brevity, we denote \((\tilde\sigma_n, \tilde\tau_n) = (\tilde\sigma^{2^n}, \tilde\tau^{2^n})\).
  We will demonstrate a bound of the form \(\tilde\tau_n - \tilde\sigma_n \ge \epsilon n^{-1}\) for some \(\epsilon > 0\) 
  and all sufficiently large \(n\) from which we can conclude non-explosion of the RDE as 
  \[\xi \ge \sum_n (\tilde\tau_n - \tilde\sigma_n) \ge \sum_n \epsilon n^{-1} = \infty.\]

  For any \(t > 0\) (to be specified later) and \(n \in \mathbb{N}\) sufficiently large 
  (so that Lemma~\ref{lem:R-est} is applicable), we introduce the sequence of times 
  \((\eta_n)_{n = 0}^{2^n}\) by setting 
  \[\eta_{k} = \inf\{t > \tilde\sigma_n : |Y_{t}| \ge 2^n + k\} \wedge (\tilde\sigma_n + t)\]
  for \(k = 0, \dots, 2^n\) so that \(\eta_0 = \tilde\sigma_n\) and \(\eta_{2^n} = \tilde\tau_n \wedge (\tilde\sigma_n + t)\).
  Moreover, we track the times \(\CT_n = \{k \in \{0, \dots, 2^n\} : \eta_{k + 1} - \eta_k > 2^{-(n + 1)}\}\).
  
  Straightaway, observing that if \(|\CT_n| \ge \frac{1}{2}2^n\), then 
  \(\tilde\tau_n - \tilde\sigma_n \ge |\CT_n| 2^{-(n + 1)} = \frac{1}{4}\) and we are done. On the other hand, 
  assuming \(|\CT_n| < \frac{1}{2}2^n\), if \(\tilde\tau_n \le \tilde\sigma_n + t\), we have
  \[\sum_{k \in \CT_n^c} |Y_{\eta_{k + 1}} - Y_{\eta_k}| = |\CT_n^c| = 2^n + 1 - |\CT_n| \ge \frac{1}{2} 2^n.\]
  Thus, in order to show \(\tilde\tau_n - \tilde\sigma_n > t\), it suffices to show that 
  \(\sum_{k \in \CT_n^c} |Y_{\eta_{k + 1}} - Y_{\eta_k}| < \frac{1}{2} 2^n\) for sufficiently large \(n\).
  By Proposition~\ref{prop:remainder} and Lemma~\ref{lem:R-est}, we observe
  \begin{align*}
    \sum_{k \in \CT_n^c} |Y_{\eta_{k + 1}} - Y_{\eta_k}| 
    & \le \sum_{k \in \CT_n^c} |R^{\mathbf{Y}, \mathbf{1}}_{\eta_k, \eta_{k + 1}}| 
        + \sum_{f \in \CF_{N - 1}}\sum_{k = 0}^{2^n - 1} |\partial^f_X \phi(Y_{\eta_k}) \<\mathbf{X}_{\eta_k, \eta_{k + 1}}, [f]\>|\\
    & \le K 2^{(\alpha + \delta) n}2^{-(n + 1)N\gamma} |\CT_n^c|
      + \sum_{f \in \CF_{N - 1}}\sum_{k = 0}^{2^n - 1} |\partial^f_X \phi(Y_{\eta_k}) \<\mathbf{X}_{\eta_k, \eta_{k + 1}}, [f]\>|
  \end{align*}
  where we choose any \(\delta \in (0, \frac{\gamma}{N})\) small enough that \(\alpha + \delta < N\gamma\) (which is possible since \(\alpha < N\gamma\)).

  Straightaway, as \(K 2^{(\alpha + \delta) n}2^{-(n + 1)N\gamma} |\CT_n^c| \lesssim 2^{-(N\gamma - (\alpha + \delta)) n} 2^n\),
  it is bounded by \(\frac{1}{4} 2^n\) for sufficiently large \(n\). Namely, it remains to bound the 
  sum in the above estimate by \(\frac{1}{4} 2^n\). Fixing \(f \in \CF_{N - 1}\), we have by 
  Corollary~\ref{cor:partial-bound3} that
  \begin{align*}
    & \sum_{k = 0}^{2^n - 1} |\partial^f_X \phi(Y_{\eta_k}) \<\mathbf{X}_{\eta_k, \eta_{k + 1}}, [f]\>|\\
    \le\ & |\partial^f_X \phi(Y_{\eta_0}) X_{\eta_0, \eta_{2^n}}| + 
      \left||\partial^f_X\phi(Y_{\eta_0})X_{\eta_0, \eta_{2^n}}|
            - \sum_{k = 0}^{2^n - 1} |\partial^f_X\phi(Y_{\eta_k}) \<\mathbf{X}_{\eta_k, \eta_{k + 1}}, [f]\>|\right|\\
    \le\ & c' (1 + 2^n) t^{\gamma} + \sum_{k = 0}^{2^n - 1} 
      \left||\partial^f_X \phi(Y_{\eta_k})| - |\partial^f_X \phi(Y_{\eta_{k + 1}})|\right||\<\mathbf{X}_{\eta_{k}, \eta_{2^n}}, [f]\>|
  \end{align*}
  for some constant \(c' > 0\) independent of \(n\). Thus, by choosing \(t \le (16c'|\CF_{N - 1}|)^{-\frac{1}{\gamma}}\),
  the first term of the above estimate is bounded by \(\frac{1}{8|\CF_{N - 1}|} 2^n\). On the other hand, 
  \begin{equation}\label{eq:del-Y-est2-fin}
    \begin{split}
      & \sum_{k = 0}^{2^n - 1} 
        \left||\partial^f_X \phi(Y_{\eta_k})| - |\partial^f_X \phi(Y_{\eta_{k + 1}})|\right||\<\mathbf{X}_{\eta_{k}, \eta_{2^n}}, [f]\>|\\
      \le\ &\sum_{k = 0}^{2^n - 1} 
        \|D\partial^f_X \phi(Y)\|_{\infty; I_r}|Y_{\eta_k} - Y_{\eta_{k + 1}}| \|\mathbf{X}\|_{\gamma} t^{(|f| + 1)\gamma}
      \lesssim 
        (\log l(2^n))^{\gamma \Lambda_1([f])} t^{\gamma (|f| + 1)} 2^n.  
    \end{split}
  \end{equation}
  Hence, as \(\Lambda_1([f]) \le |f| + 1\), by taking \(t \le \frac{\epsilon}{n}\) 
  for some fixed \(\epsilon > 0\) small enough, the right hand side of the above estimate is bounded by 
  \((8|\CF_{N - 1}|)^{-1} 2^n\). Consequently, combining the above estimates, we have shown 
  \[\sum_{k \in \CT_n^c} |Y_{\eta_{k + 1}} - Y_{\eta_k}| 
    < \frac{1}{4}2^n + \sum_{f \in \CF_{N - 1}} \left(\frac{1}{8|\CF_{N - 1}|}2^n + \frac{1}{8|\CF_{N - 1}|}2^n\right) = \frac{1}{2}2^n\]
  for sufficiently large \(n\) and \(t = \frac{\epsilon}{n} \wedge \frac{1}{\sqrt[\gamma]{16c'}}\) as claimed.
\end{proof}

We remark that the assumption that \(Y\) is one-dimensional is only used in order to ensure 
\(|Y_{\eta_k} - Y_{\eta_{k + 1}}| = 1\) to obtain the final inequality in \eqref{eq:del-Y-est2-fin}. 
Thus, the above theorem remains to hold for higher dimensional bRDEs if we can control the terms
\(\left||\partial^f_X \phi(Y_{\eta_k})| - |\partial^f_X \phi(Y_{\eta_{k + 1}})|\right|\)
in \eqref{eq:del-Y-est2-fin} directly. One way to achieve this is to instead assume an oscillation bound 
on \(\phi\) of the form
\[\Osc_r(\phi) := \sup_{\|x\| \le r} \phi(x) - \inf_{\|x\| \le r} \phi(x) \lesssim (\log l(r))^\gamma.\]

\begin{example}
  In the case \(\mathbf{X}\) is the rough path (It\^o or Stratonovich) 
  corresponding to a Brownian motion \(B\), we may take \(\gamma\) arbitrarily close to \(\frac{1}{2}\). Thus, the 
  constant \(\alpha\) in Assumption~\ref{as:growth2} can be taken to be arbitrarily close to 1. Namely, 
  Assumption~\ref{as:growth2} is satisfied if there exists some \(\epsilon > 0\),
  \begin{itemize}
    \item \(\|D\phi(x)\| \le (\log l(\|x\|))^{\frac{1}{2} - \epsilon}\),
    \item \(\|D^n\phi(x)\| \le l(\|x\|)^{(1 - n)(1 - \epsilon)}\) for \(n = 0\) and \(n = 2\).
  \end{itemize}
  Consequently, under this growth condition, the SDE (either It\^o or Stratonovich)
  \[\dd X_t = \phi(X_t) \dd B_t\]
  is strongly complete.
\end{example}

\begin{example}
  In general where \(\mathbf{X} \in \CC^\gamma([0, T], \mathbb{R}^d)\) for some 
  \(\gamma \in (\frac{1}{N + 1}, \frac{1}{N}]\), in contrast to the Brownian case, a naïve 
  application of the above theorem fails to achieve almost linear growth for \(\phi\). Namely, 
  condition~\ref{as:growth2} will assert that \(\|\phi(x)\| \le l(\|x\|)^{N\gamma}\) for which 
  \(N \gamma < 1\). Nonetheless, we can recover an almost linear growth condition by dropping a level 
  lower. 

  For any \(\gamma' \in (\frac{1}{N + 2}, \frac{1}{N + 1}]\), we have the embedding 
  \[\mathbf{X} \mapsto \mathbf{X}' : \CC^\gamma([0, T], \mathbb{R}^d) \hookrightarrow \CC^{\gamma'}([0, T], \mathbb{R}^d)\] 
  where for any \(f \in \CF_{N}\), we define \(\<\mathbf{X}'_{st}, f\> = \<\mathbf{X}_{st}, f\>\) and
  \[\<\mathbf{X}'_{st}, [f]_a\> := \int_s^t \<\mathbf{X}_{sr}, f\> \dd X_r^a\]
  where the latter integral can be understood in the Young sense. 
  
  Consequently, for any \(\epsilon > 0\), by taking \(\gamma' \in (\frac{1 - \epsilon}{N + 1}, \frac{1}{N + 1})\),  
  Assumption~\ref{as:growth2} is satisfied if 
  \(\phi \in C^{N + 2}(\mathbb{R}; \CL(\mathbb{R}^d, \mathbb{R}))\) is such that 
  \begin{itemize}
    \item \(\|D\phi(x)\| \le (\log l(\|x\|))^{\gamma'}\),
    \item \(\|D^n\phi(x)\| \le l(\|x\|)^{(1 - n)(1 - \epsilon)}\) for all \(n = 0, 2, 3 \dots, N + 1\).
  \end{itemize}
  Thus, we recover that the RDE~\eqref{eq:bRDE} has a unique global solution while allowing for 
  \(\phi\) to have almost linear growth and unbounded derivative.
\end{example}

\section{Sharpness of Theorem~\ref{thm:non-explosion}}\label{sec:sharpness}
\subsection{Pure area branched rough paths}\label{subsec:pure-area}

As our counterexamples utilize pure area branched rough paths (i.e. 
branched rough paths whose path component \(\<\mathbf{X}, \tree<X'>\>\) is identically zero), 
we first in this subsection provide some useful facts about them. 

In the case where \(N = 2\), it is well-known (cf. \cite{Friz:20}) that the L\'evy area 
of any pure area rough paths is of the form \(\delta f\) for some \(f \in C^{2\gamma}\). 
By computing the coboundary of \(\<\mathbf{X}, \tau\>\) for \(\tau \in \{\tree<2'>, \tree<10'>\}\) directly, this 
fact seemingly also extends to the case where \(N = 3\). Namely, \(\mathbf{X}\) is a \(\gamma\)-H\"older branched 
rough path with \(\gamma \in (\frac{1}{4}, \frac{1}{3}]\) and \(\<\mathbf{X}, \tree<X'>\> = 0\) if and only if
\begin{equation}\label{eq:pure-area-3}
  \mathbf{X}_{t} = \mathbf{1} + f^{\tree<1'>}(t) \tree<1'> 
    + f^{\tree<2'>}(t) \tree<2'> + f^{\tree<10'>}(t) \tree<10'>,
\end{equation}
for some \(f^{\tree<1'>} \in C^{2\gamma}([0, T], (\mathbb{R}^d)^{\otimes 2}), 
f^{\tree<2'>}, f^{\tree<10'>} \in C^{3\gamma}([0, T], (\mathbb{R}^d)^{\otimes 3})\) and we have that
\[\mathbf{X}_{st} \equiv \mathbf{X}_s^{-1} \curvearrowright \mathbf{X}_t = \mathbf{1} + (\delta f^{\tree<1'>})_{st} \tree<1'> 
  + (\delta f^{\tree<2'>})_{st} \tree<2'> + (\delta f^{\tree<10'>})_{st} \tree<10'>.\]
Note that here, testing \(\mathbf{X}_{t}\) against a labelled tree means taking the coordinate 
of the tensor component of the corresponding function, e.g. 
\(\<\mathbf{X}_{t}, \tree<2>\> = (f^{\tree<2'>}(t))_{abc}\).
We remark that \(\mathbf{X}\) defined by \eqref{eq:pure-area-3} is indeed a branched rough path 
as it is simply the level 3 truncation of \(\exp(f^{\tree<1'>}(t) \tree<1'> 
    + f^{\tree<2'>}(t) \tree<2'> + f^{\tree<10'>}(t) \tree<10'>)\).

However, it turns out this fact does not extend to the case where \(N \geq 4\). Indeed, truncating 
\(\exp(t \tree<1'>)\) at level \(N = 4\) on \(\mathbb{R}\), we obtain the pure area branched rough path defined by 
\[\mathbf{X}_t = \mathbf{1} + t\ \tree<1'> + \frac{t^2}{2}\left(\tree<1'>^2 + \tree<02'> + \tree<100'>\hspace{0.6mm} \right).\]
Then, for any \(s < u < t\), one may compute that
\[(\delta \langle\mathbf{X}, \tree<100'>\rangle)_{sut} = \<\mathbf{X}, \tree<1'>\>_{su}\<\mathbf{X}, \tree<1'>\>_{ut}
  = (u - s)(t - u) \neq 0\]
implying that \(\langle\mathbf{X}_{st}, \tree<100'>\rangle\) is not a coboundary.
Instead, we have the following weaker characterization of pure area branched rough paths.
\begin{lemma}
  Let \(N \in \mathbb{N}\), \(\gamma \in (\frac{1}{N + 1}, \frac{1}{N}]\) and \(\mathbf{X}\) be 
  a \(\gamma\)-H\"older branched rough path. Then, for any \(\tau \in \CT_N\), we have that 
  \[\<\mathbf{X}_{st}, \tau\> = \delta f_{st}\]
  for some \(f \in C^{\gamma |\tau|}([0, T], (\mathbb{R}^d)^{\otimes |\tau|})\) if 
  \(\<\mathbf{X}_t, \sigma\> = 0\) for all \(\sigma \in \CT \setminus \{\mathbf{1}\}\) with 
  \(|\sigma| \le \left\lfloor\frac{|\tau|}{2}\right\rfloor\).
\end{lemma}
\begin{proof}
  We will show that \(\<\mathbf{X}_{st}, \tau\> = \delta \<\mathbf{X}, \tau\>_{st}\) for which 
  the regularity condition follows directly from the regularity of the branched rough path \(\mathbf{X}\).

  Recalling~\eqref{eq:coproduct-sum-subtrees} that the coproduct \(\Delta\) of a tree \(\tau\) can be 
  expressed as
  \[\Delta \tau = \tau \otimes \mathbf{1} + \mathbf{1} \otimes \tau + 
    \sum_{\sigma \subsetneq \tau} (\tau \setminus \sigma) \otimes \sigma\]
  where the sum is over all proper rooted subtrees \(\sigma\) of \(\tau\), we have that
  \begin{align*}
    \<\mathbf{X}_{st}, \tau\> & = \<\mathbf{X}_s^{-1} \curvearrowright \mathbf{X}_t, \tau\> 
      = \<\mathbf{X}_s^{-1} \otimes \mathbf{X}_t, \Delta \tau\> \\
    & = \<\mathbf{X}_t, \tau\> + \<\mathbf{X}_s, S \tau\>
      + \sum_{\sigma \subsetneq \tau} \<\mathbf{X}_s, S(\tau \setminus \sigma)\> \<\mathbf{X}_t, \sigma\>.
  \end{align*}
  On the other hand, by observing that for any \(\tau \in \CT\),
  \begin{align*}
    0 & = \nabla (S \otimes \text{id}) \Delta \tau
        = \nabla (S \otimes \text{id}) \left(\tau \otimes \mathbf{1} + \mathbf{1} \otimes \tau + 
            \sum_{\sigma \subsetneq \tau} (\tau \setminus \sigma) \otimes \sigma\right)\\
      & = S \tau + \tau + \sum_{\sigma \subsetneq \tau} S(\tau \setminus \sigma) \sigma,
  \end{align*}  
  it follows by a simple induction argument that 
  \[S \tau = -\tau + \sum_{f \in \CI_{\tau}} a_f f\]
  for some \(\CI_{\tau} \subseteq \CF_{|\tau|} \cap \{f \in \CF : 1 < \# f\}\) and \(a_f \in \mathbb{Z}\) for each \(f \in \CI_{\tau}\).
  Thus, we have that
  \begin{align*}
    \<\mathbf{X}_{st}, \tau\> 
    =\ & \<\mathbf{X}_t, \tau\> - \<\mathbf{X}_s, \tau\>
        + \sum_{f \in \CI_{\tau}} a_f \<\mathbf{X}_s, f\>\\
    &  + \sum_{\sigma \subsetneq \tau} \<\mathbf{X}_s, - \tau \setminus \sigma + 
            \sum_{f \in \CI_{\tau \setminus \sigma}} a_f f\> \<\mathbf{X}_t, \sigma\>.
  \end{align*}
  Now, as for any \(f = \prod_{i = 1}^k f_i \in \CI_\tau, k \ge 2\), there exists \(i\) such that 
  \(|f_i| \le \left\lfloor\frac{|\tau|}{2}\right\rfloor\), by the assumption on \(\mathbf{X}\), it follows that
  \(\<\mathbf{X}_s, f\> = \prod_{i = 1}^k \<\mathbf{X}_s, f_i\> = 0\)
  and the first sum above vanishes. Similarly, by the same reason, the second sum also vanishes and thus,
  \[\<\mathbf{X}_{st}, \tau\> = \<\mathbf{X}_t, \tau\> - \<\mathbf{X}_s, \tau\> = \delta \<\mathbf{X}, \tau\>_{st}\]
  as claimed.
\end{proof}

\begin{corollary}\label{cor:pure-area}
  Let \(N \in \mathbb{N}\), \(\gamma \in (\frac{1}{N + 1}, \frac{1}{N}]\) and 
  \(\mathbf{X} : [0, T] \to G_N(\CH)\) be such that \(\<\mathbf{X}, \sigma\> = 0\) for all 
  \(\sigma \in \CT_{\left\lfloor\frac{N}{2}\right\rfloor} \setminus \{\mathbf{1}\}\). 
  Then, \(\mathbf{X}\) is a \(\gamma\)-H\"older branched rough path if and only if 
  for any \(\tau \in \CT_N \setminus \CT_{\left\lfloor\frac{N}{2}\right\rfloor}\), there 
  exists \(f^\tau \in C^{\gamma|\tau|}([0, T], (\mathbb{R}^d)^{\otimes |\tau|})\) such that
  \[\mathbf{X}_t = \mathbf{1} + \sum_{\tau \in \CT_N \setminus \CT_{\left\lfloor\frac{N}{2}\right\rfloor}} f^\tau(t) \tau.\]
  Moreover, we have that 
  \[\mathbf{X}_{st} \equiv \mathbf{X}_s^{-1} \curvearrowright \mathbf{X}_t 
    = \mathbf{1} + \sum_{\tau \in \CT_N \setminus \CT_{\left\lfloor\frac{N}{2}\right\rfloor}} (\delta f^\tau)_{st} \tau.\]
\end{corollary}
\begin{proof}
  The forward implication is immediate by the previous lemma while the backward implication 
  follows directly by observing that \(\mathbf{X}_t\) is the level \(N\) truncation of
  \(\exp\left(\sum_{\tau \in \CT_N \setminus \CT_{\left\lfloor\frac{N}{2}\right\rfloor}} f^\tau(t) \tau\right)\).
\end{proof}

\subsection{Examples of branched RDEs which explode in finite time}

Leveraging on the construction of the pure area branched rough paths above, we in this subsection
provide some examples of branched RDEs of level \(N\) with vector fields with growth of
polynomial order \(\frac{1}{N} + \epsilon\) for which the solution explodes in finite time.
In particular, these examples demonstrate that the growth condition in Theorem~\ref{thm:non-explosion}
is essentially sharp in the case where \(\gamma = \frac{1}{N}\) and \(\kappa = 0\).

For any multi-index \(\alpha = (a_0, \dots, a_n)\), we take \(f_\alpha \in C^n_b(\mathbb{R}; \mathbb{R})\)
such that \(f_\alpha^{(k)}(0) = a_k\) for any \(0 \le k \le n\). As an example, one may take
\(f_\alpha(x) = \tilde f_\alpha(x) \chi(x)\) with \(\tilde f_\alpha(x) := \sum_{k = 0}^n a_k x^k\)
and \(\chi\) a smooth cutoff function satisfying \(\chi(x) = 1\) for \(|x| \le 1\) and \(\chi(x) = 0\)
for \(|x| \ge 2\). Moreover, we will fix some \(\epsilon, \delta > 0\) with \(\epsilon > (N - 1)\delta\).

We first consider the vector field \(\phi : \mathbb{R}^2 \to \mathbb{R}^2\) defined by
\[\phi(x) := (x_1^{\frac{1}{N} + \epsilon} f_{e_{N - 1}}(x_2), x_1^{\frac{1}{N} - \delta} f_{e_0}(x_2))^T\]
where \(e_k \in \mathbb{R}^{N}\) denotes the multi-index with \(1\) at the \(k + 1\)-th position and \(0\)
elsewhere\footnote{We shifted the numeration by one so that the \(k\)-th index aligns with the
\(k\)-th derivative}. It is easy to see that \(\phi\) and its derivatives have polynomial growth of order
\(\frac{1}{N} + \epsilon\) and \(\frac{1}{N} - \delta\) respectively. Then, we consider the branched RDE
\[\dd \mathbf{Y}_t = \phi(\mathbf{Y}_t) \dd \mathbf{X}_t,\ Y_0 = (1, 0)^T \in \mathbb{R}^2\]
driven by the one-dimensional pure area branched
rough path \(\mathbf{X}_t\) defined by
\[\mathbf{X}_t = \mathbf{1} + t[(\tree<X'>)^{N - 1}] = \mathbf{1} + t \overbrace{\tree<paw'>}^{N - 1}.\]
We remark that \(\mathbf{X}\) is well-defined as a branched rough path by Corollary~\ref{cor:pure-area}.
By Equation~\eqref{eq:deriv-comp-def} and the definition of the branched rough integral,
the branched RDE is equivalent to the ODE defined by
\[\dot Y_t = \frac{1}{(N - 1)!} D^{N - 1}\phi(Y_t)[(\phi(Y_t))^{\otimes (N - 1)}],\ Y_0 = (1, 0)^T \in \mathbb{R}^2.\]
Namely, written component-wise, we have that
\[\dot Y_t^k = \frac{1}{(N - 1)!} \sum_{i = 1}^{N - 1} \sum_{j_i = 1}^2 \partial_{j_1, \ldots, j_{N - 1}}^{N - 1}
    \phi_k(Y_t) \phi_{j_1}(Y_t) \cdots \phi_{j_{N - 1}}(Y_t),\ k = 1, 2.\]
Now, observing that for any \(x = (x_1, 0)^T\) with \(x_1 > 0\),
\(\phi(x) = (0, x_1^{\frac{1}{N} - \delta})^T\), we have that
\[\sum_{i = 1}^{N - 1} \sum_{j_i = 1}^2 \partial_{j_1, \ldots, j_{N - 1}}^{N - 1}
    \phi_k(Y_t) \phi_{j_1}(Y_t) \cdots \phi_{j_{N - 1}}(Y_t)
  = x_1^{(N - 1)\left(\frac{1}{N} - \delta\right)} \partial^{N - 1}_{2, \ldots, 2} \phi_k(x).\]
Consequently, it follows that the ODE is equivalent to the system
\(\dot Y_t^1 = \frac{1}{(N - 1)!}(Y^1_t)^{1 + (\epsilon - (N - 1)\delta)}\) and \(\dot Y_t^2 = 0\) with initial condition \(Y_0 = (1, 0)^T\)
which explodes in finite time as \(\epsilon - (N - 1)\delta > 0\).

\vspace{2mm}
We provide another example: Define the vector field \(\phi : \mathbb{R}^N \to \mathbb{R}^N\) by
\[\phi(x)_k := \begin{cases}
  x_1^\beta f_{e_1}(x_{k + 1}) & k < N, \\
  x_1^\beta f_{e_0}(x_N) & k = N,
\end{cases}\]
for some \(\beta > \frac{1}{N}\).
Again, by Corollary~\ref{cor:pure-area}, we may construct the pure area branched rough path on
\(\mathbb{R}\) by setting
\[\mathbf{X}_t = \mathbf{1} + t \underbrace{\tree<lin'>}_{N} = \mathbf{1} + t [-]^N(\mathbf{1})\]
where \([-]^N\) denotes the \(N\)-fold application of the rooting operator \([-]\).
We claim that the branched RDE
\begin{equation}\label{eq:exploding-rde}
  \dd \mathbf{Y}_t = \phi(\mathbf{Y}_t) \dd \mathbf{X}_t
\end{equation}
with initial condition \((1, 0, \dots, 0) \in \mathbb{R}^N\) explodes in finite time.

Indeed, by Equation~\eqref{eq:deriv-comp-def} and the definition of the branched rough integral,
the branched RDE is equivalent to the system of ODEs
\[\dd Y_t = (D\phi(Y_t))^{N - 1} \phi(Y_t) \dd t\]
with the same initial condition. Moreover, by the construction of \(f_\alpha\), we observe that, for any \(x_1 > 0\),
\(x = (x_1, 0, \dots, 0)\)
\[D\phi(x) = \begin{pmatrix}
0 & x_1^{\beta} & 0 & \cdots & 0 \\
0 & 0 & x_1^{\beta} & \cdots & 0 \\
\vdots & \vdots & \ddots & \ddots & \vdots \\
0 & 0 & \cdots & 0 & x_1^{\beta} \\
\beta x_1^{\beta-1} & 0 & \cdots & 0 & 0
\end{pmatrix},\]
from which we obtain \((D\phi(x))^{N - 1} \phi(x) = (x_1^{N\beta}, 0, \dots, 0)^T\). Thus,
the RDE~\eqref{eq:exploding-rde} is equivalent to the ODE
\(\dd Y_t^1 = (Y_t^1)^{N\beta} \dd t, Y_0^1 = 1\) and \(Y_t^k = 0\) for all \(k > 1\) which explodes in
finite time as \(N\beta > 1\).

\appendix
\section{Proof of Lemma~\ref{lem:R-est}}\label{sec:appendixA}
We will in this section complete the proof of Theorem~\ref{thm:non-explosion2} by estimating the 
\(N\gamma\)-H\"older norm of the remainder \(R^{\mathbf{Y}, \mathbf{1}}\) using a similar strategy as  
presented in Section~\ref{sec:non-exp}.

Throughout this section, we denote \(I_r = [\tilde\sigma^r, \tilde\tau^{r}]\) with \(\tilde\sigma^r\) and \(\tilde\tau^{r}\) 
defined as in Equation~\eqref{eq:tau-sigma-r} for some \(r > 0\). Moreover, we 
take \(I = I_\epsilon \subseteq I_r\) to be an interval of size smaller than \(\epsilon > 0\).

\begin{lemma}\label{lem:l-prop}
  For \(l : \mathbb{R}_+ \to [1, \infty)\) as described in Assumption~\ref{as:growth2}, for any \(r > 0\), 
  there exists a constant \(c > 0\) such that
  \[cl(2^{n + 1} r) \le l(2^n r) \le l(2^{n + 1}r).\]
\end{lemma}
\begin{proof}
  Straightaway, as \(l\) is non-decreasing, \(l(2^n r) \le l(2^{n + 1}r)\) and it remains to show that there exists 
  some \(c >0\) such that \(cl(2^{n + 1} r) \le l(2^n r)\). 

  As \(\int_0^\infty \frac{1}{l(t)} \dd t \ge \int_0^\infty c^{-1}(1 + r)^{-1} \dd r = \infty\), we have by the integral test and 
  the Cauchy condensation test that \(\sum_n a_n := \sum_n \frac{2^n r}{l(2^n r)} = \infty\). Now, as 
  \[\left|\frac{a_{n + 1}}{a_n}\right| = \left|\frac{2^{n + 1}r}{l(2^{n + 1} r)} \frac{l(2^n r)}{2^n r}\right| 
    = 2\left|\frac{l(2^n r)}{l(2(2^n r))}\right|\]
  and the map \(x \mapsto \frac{l(2x)}{l(x)}\) is monotone, the limit 
  \(\lim_{n \to \infty}\left|\frac{a_{n + 1}}{a_n}\right|\) exists. Thus, by the ratio test, 
  \(\lim_{n \to \infty}\left|\frac{l(2^n r)}{l(2(2^n r))}\right| \ge \frac{1}{2}\) providing the 
  desired inequality. 
\end{proof}

With this, we may estimate the derivatives of \(\phi(Y)\) on \(I_r\).

\begin{corollary}\label{cor:partial-bound3}
  For any \(f \in \CF_{N - 1}\) and \(n \le N - |f|\), we have that 
  \begin{equation}\label{eq:partial-bound3-a}
    \|D^n \partial^f_X \phi(Y)\|_{\infty; I_r} \lesssim l(r)^{(1 - n)\alpha + \delta}
  \end{equation}
  for any \(r, \delta > 0\). Moreover, in the special case where \(n = 1\), we have that 
  \begin{equation}\label{eq:partial-bound3-b}
    \|D\partial^f_X \phi(Y)\|_{\infty; I_r} \lesssim (\log l(r))^{\gamma \Lambda_1([f])}. 
  \end{equation}
\end{corollary}
\begin{proof}
  Equation~\eqref{eq:partial-bound3-a} follows directly from \eqref{eq:partial-bound3} 
  and Lemma~\ref{lem:l-prop} and it remains to show \eqref{eq:partial-bound3-b}. To this 
  end, we will first prove 
  \begin{equation}\label{eq:partial-bound3-c}
    \|\partial^f_X \phi(Y)\|_{\infty; I_r} \lesssim (\log l(r))^{\gamma \Lambda_1([f])} l(r)^\alpha
  \end{equation}
  for any \(f \in \CF_{N - 1}\). This is clearly true for \(f = \mathbf{1}\) and so, taking 
  \(f = \prod_{i = 1}^k [h_i]\) and assuming the relevant inductive hypothesis, it suffices 
  to show \eqref{eq:partial-bound3-c} for \(f\). For this, we consider the case where \(k = 1\) 
  and \(k > 1\) separately. For \(k = 1\), \(f = [h]\) and we have that
  \begin{align*}
    \|\partial^f_X \phi(Y)\|_{\infty; I_r} 
    & \lesssim \|D\phi(Y)\|_{\infty; I_r} \|\partial^h_X \phi(Y)\|_{\infty; I_r}\\
    & \lesssim \log(l(r))^{\gamma}(\log l(r))^{\gamma \Lambda_1([h])} l(r)^\alpha 
      = (\log l(r))^{\gamma \Lambda_1([f])} l(r)^\alpha. 
  \end{align*}
  On the other hand, for \(k > 1\), we have that
  \begin{align*}
    \|\partial^f_X \phi(Y)\|_{\infty; I_r} 
    & \lesssim \|D^k\phi(Y)\|_{\infty; I_r} \prod_{i = 1}^k \|\partial^{h_i}_X \phi(Y)\|_{\infty; I_r}\\
    & \lesssim l(r)^{(1 - k)\alpha} \prod_{i = 1}^k (\log l(r))^{\gamma \Lambda_1([h_i])} l(r)^\alpha
      = (\log l(r))^{\gamma \Lambda_1([f])} l(r)^\alpha
  \end{align*}
  as claimed.

  With this in mind, we may now prove \eqref{eq:partial-bound3-b}. Again, the statement is clearly true
  for \(f = \mathbf{1}\) and so, we take \(f = \prod_{i = 1}^k [h_i]\) and assume the relevant 
  inductive hypothesis. Then, we have by using \eqref{eq:partial-bound3-c} that
  \begin{align*}
    \|D\partial^f_X \phi(Y)\|_{\infty; I_r} 
    \lesssim\ & \|D^{k + 1}\phi(Y)\|_{\infty; I_r} \prod_{i = 1}^k\|\partial^{h_i}_X \phi(Y)\|_{\infty; I_r}\\
    & + \|D^k \phi(Y)\|_{\infty; I_r} \sum_{j = 1}^k \|D \partial^{h_j}_X \phi(Y)\|_{\infty; I_r} 
          \prod_{i \neq j} \|\partial^{h_i}_X \phi(Y)\|_{\infty; I_r}\\
    \lesssim\ & l(r)^{-k\alpha} \prod_{i = 1}^k (\log l(r))^{\gamma \Lambda_1([h_i])} l(r)^\alpha\\
    & + l(r)^{(1 - k)\alpha} \sum_{j = 1}^k (\log l(r))^{\gamma \Lambda_1([h_j])} 
          \prod_{i \neq j} (\log l(r))^{\gamma \Lambda_1([h_i])} l(r)^\alpha\\
    \sim\ & \log(l(r))^{\gamma \Lambda_1([f])}
  \end{align*}
  as desired.
\end{proof}

We remark that, in the above and the subsequent estimates, the constant within the inequality 
indicated by \(\lesssim\) depends on the choice of \(\delta\) but not on \(r\).

\begin{lemma}\label{lem:R-sharp-bd2}
  Under Assumption~\ref{as:growth2}, we have
  \[\epsilon^{\alpha}\|R^{\sharp, f}\|_{N\gamma; I} \lesssim
    \sum_{k = 0}^{N - |f| - 1}\epsilon^{k\alpha - \delta}\|R^{\mathbf{Y}, \mathbf{1}}\|_{N\gamma; I}^{k}\]
  for any \(\delta > 0\) and \(\epsilon \le l(r)^{-1}\).
\end{lemma}
\begin{proof}
  As in the proof of Lemma~\ref{lem:R-sharp-bd}, we estimate each summand in \(R^{\sharp, f}\) for 
  a given \(f\). In particular, fixing \(\delta' = \frac{\delta}{2N}\), 
  \(f \in \CF_{N - 1}, k \in \{1, \dots, N - |f| - 1\}\) and \(m_1 + m_2 + m_3 = k\) with \(m_1 \neq k\), 
  we have by Equation~\eqref{eq:remain-Y-one} and Corollary~\ref{cor:partial-bound3} 
  (where in the latter equation we replace \(\delta\) by \(\delta'\)) that
  \begin{equs}\label{eq:R-sharp-bd2}
      & \epsilon^{\alpha} \|D^k[\partial^f_X \phi(Y)]\|_{\infty; I} \|A^{m_1, m_2, m_3}\|_{N\gamma; I}\\[2ex]
      \lesssim\ 
      & 
      \epsilon^{\alpha} l(r)^{(1 - k)\alpha + \delta'}
      \|R^{\mathbf{Y}, \mathbf{1}}\|_{N\gamma; I}^{m_1 + m_2} \epsilon^{(m_1 + m_2 - 1)N\gamma}
      \sum_{|\tau| > N - k} 
      \|\partial^\tau_X Y\|_{\infty; I}^{m_3}
      \epsilon^{m_3|\tau|\gamma}\\
      & + 
      \epsilon^{\alpha} l(r)^{(1 - k)\alpha + \delta'}
      \|R^{\mathbf{Y}, \mathbf{1}}\|_{N\gamma; I}^{m_2} \epsilon^{(m_2 - 1)N\gamma}
        \sum_{\substack{|\tau| > N - k\\ \sigma \in \CT_{N - 1}^+}}
          \|\partial_X^\sigma Y\|^{m_1}_{\infty; I}\|\partial_X^\tau Y\|^{m_3}_{\infty; I}
          \epsilon^{(m_1|\sigma| + m_3|\tau|)\gamma}
  \end{equs}
  for which we estimate each of the two terms separately. Estimating the first term, we observe 
  \begin{equation}\label{eq:R-sharp-bd-term-a}
    \begin{split}
      & \epsilon^{\alpha} l(r)^{(1 - k)\alpha + \delta'}
        \|R^{\mathbf{Y}, \mathbf{1}}\|_{N\gamma; I}^{m_1 + m_2} \epsilon^{(m_1 + m_2 - 1)N\gamma}
        \sum_{|\tau| > N - k} 
        \|\partial^\tau_X Y\|_{\infty; I}^{m_3}
        \epsilon^{m_3|\tau|\gamma}\\
      \lesssim\
      & \|R^{\mathbf{Y}, \mathbf{1}}\|_{N\gamma; I}^{m_1 + m_2} 
        \epsilon^{\alpha + (m_1 + m_2 - 1)N \gamma + m_3(N - k + 1)\gamma}
        l(r)^{(1 - k)\alpha + \delta' + m_3(\alpha + \delta')}\\
      \le\ & \epsilon^{(m_1 + m_2)\alpha - N\delta'} \|R^{\mathbf{Y}, \mathbf{1}}\|_{N\gamma; I}^{m_1 + m_2} 
        \epsilon^{(m_1 + m_2 - 1)(N\gamma - \alpha) + m_3(N - k + 1)\gamma} l(r)^{(1 - k + m_3)\alpha}
    \end{split} 
  \end{equation}
  In the case where \(m_1 + m_2 \ge 1\) (namely \(1 - k + m_3 \le 0\)), 
  we observe that the term on the right hand side of \eqref{eq:R-sharp-bd-term-a} is bounded by
  \(\epsilon^{(m_1 + m_2)\alpha - N\delta'}\|R^{\mathbf{Y}, \mathbf{1}}\|_{N\gamma; I}^{m_1 + m_2}\).
  On the other hand, if \(m_1 + m_2 = 0\), i.e. \(m_3 = k\), the term becomes
  \[\epsilon^{-N\delta' - (N\gamma - \alpha) + k(N - k + 1)\gamma}l(r)^\alpha 
    \lesssim \epsilon^{(-N + k(N - k + 1))\gamma}\epsilon^{-N \delta'} =: \epsilon^{P_N(k)\gamma}\epsilon^{- N\delta'}.\]
  Now, by observing that \(P_N(k)\) is a quadratic polynomial with negative leading coefficient and 
  roots \(1\) and \(N\), it follows that \(P_N(k) \ge 0\) for all \(k \in \{1, \dots, N\}\). Thus, 
  \(\epsilon^{P_N(k) \gamma} \lesssim 1\) and we have
  \begin{align*}
    & \epsilon^{\alpha} l(r)^{(1 - k)\alpha + \delta'}
    \|R^{\mathbf{Y}, \mathbf{1}}\|_{N\gamma; I}^{m_1 + m_2} \epsilon^{(m_1 + m_2 - 1)N\gamma}
    \sum_{|\tau| > N - k} 
    \|\partial^\tau_X Y\|_{\infty; I}^{m_3}
    \epsilon^{m_3|\tau|\gamma}\\
    \lesssim\ &
    \epsilon^{(m_1 + m_2)\alpha - N\delta'}\|R^{\mathbf{Y}, \mathbf{1}}\|_{N\gamma; I}^{m_1 + m_2}.
  \end{align*}

  We now turn to the second term in \eqref{eq:R-sharp-bd2}. Observe that 
  \begin{align*}
    & \epsilon^{\alpha} l(r)^{(1 - k)\alpha + \delta'}
      \|R^{\mathbf{Y}, \mathbf{1}}\|_{N\gamma; I}^{m_2} \epsilon^{(m_2 - 1)N\gamma}
        \sum_{\substack{|\tau| > N - k\\ \sigma \in \CT_{N - 1}^+}}
          \|\partial_X^\sigma Y\|^{m_1}_{\infty; I}\|\partial_X^\tau Y\|^{m_3}_{\infty; I}
          \epsilon^{(m_1|\sigma| + m_3|\tau|)\gamma}\\
    \lesssim\
        & \|R^{\mathbf{Y},\mathbf{1}}\|_{N\gamma; I}^{m_2} 
        \epsilon^{\alpha + (m_2 - 1)N\gamma + (m_1 + m_3 (N - k + 1))\gamma}
        l(r)^{(1 - k)\alpha + \delta' + (m_1 + m_3)(\alpha + \delta')}\\
    \lesssim\ 
        & \epsilon^{m_2\alpha - 2N\delta'} \|R^{\mathbf{Y}, \mathbf{1}}\|_{N\gamma; I}^{m_2}
          \epsilon^{(m_2 - 1)(N\gamma - \alpha) + (m_1 + m_3(N - k + 1))\gamma}
          l(r)^{(1 - m_2)\alpha}
  \end{align*}
  where we used the fact that \(- k + m_1 + m_3 = - m_2\). Now, similar to before, we observe 
  that for \(m_2 \ge 1\), the right hand side of the above estimate is bounded by 
  \(\epsilon^{m_2\alpha - 2N\delta'} \|R^{\mathbf{Y}, \mathbf{1}}\|_{N\gamma; I}^{m_2}\).
  Now, if \(m_2 = 0\), the term becomes
  \[\epsilon^{-2N\delta'} \epsilon^{-(N\gamma - \alpha) + (m_1 + m_3(N - k + 1))\gamma} l(r)^\alpha
    \lesssim \epsilon^{-2N\delta'}\epsilon^{(-N + k + m_3(N - k))\gamma} \lesssim \epsilon^{-2N\delta'}.\]
  where the last inequality follows as \(m_3 \ge 1\) as we had assumed \(m_1 \neq k\).
  Thus, by substituting the above estimates into~\eqref{eq:R-sharp-bd2} and summing over 
  \(k, m_1, m_2, m_3\), we have 
  \[\epsilon^{\alpha}\|R^{\sharp, f}\|_{N\gamma; I} 
    \lesssim \sum_{k = 0}^{N - |f| - 1}\epsilon^{k\alpha - 2N\delta'}\|R^{\mathbf{Y}, \mathbf{1}}\|_{N\gamma; I}^{k}
    = \sum_{k = 0}^{N - |f| - 1}\epsilon^{k\alpha - \delta}\|R^{\mathbf{Y}, \mathbf{1}}\|_{N\gamma; I}^{k}\]
  as claimed.
\end{proof}

\begin{lemma}\label{lem:h-poly-bdd2}
  For any \(\delta \le \frac{\gamma}{N}\), we denote
  \(h_\delta(\epsilon) = \epsilon^{\alpha + \delta} \|R^{\mathbf{Y}, \mathbf{1}}\|_{N\gamma; I}\).
  Then, assuming Assumption~\ref{as:growth2} we have that 
  \begin{equation}\label{eq:h-poly-bdd2}
    h_\delta(\epsilon) \lesssim 1 + \epsilon^{\delta}\sum_{k = 1}^N h_\delta(\epsilon)^k
  \end{equation}
  for all \(\epsilon \le l(r)^{-1}\).
\end{lemma}
\begin{proof}
  For simplicity, we in this proof omit the subscript from \(h_\delta(\epsilon)\). We have by 
  Lemma~\ref{lem:remain-Z-one} and~\ref{lem:R-sharp-bd2}
  that
  \begin{equation}\label{eq:h-est2}
    \begin{split}
      h(\epsilon) \lesssim\ 
      & \sum_{f \in \CF_{N - 1}} \epsilon^{(1 + |f|)\gamma}\sum_{k = 0}^{N - |f| - 1} 
        \epsilon^{k\alpha} \|R^{\mathbf{Y}, \mathbf{1}}\|_{N\gamma; I}^k\\
      & + \sum_{f \in \CF_{N - 1}} C_I^f \|R^{\mathbf{Y}, \mathbf{1}}\|_{N\gamma; I}^{N - |f|} 
            \epsilon^{(N - 1)(N - |f|)\gamma} \epsilon^{\alpha + \gamma + \delta}\\
      & + \sum_{f \in \CF_{N - 1}} C_I^f \sum_{\tau \in \CT^+_{N - 1}} \|\partial^\tau_X Y\|_{\infty; I}^{(N - |f|)}
            \epsilon^{(N - |f|)(|\tau| - 1)\gamma}\epsilon^{\alpha + \gamma + \delta}\\
      & + \sum_{|f| = N - 1} \|\partial^f_X \phi(\mathbf{Y})\|_{\infty; I} \epsilon^{\alpha + \delta}.
    \end{split}
  \end{equation}
  Straightaway, the summand of the first term can be estimated as
  \begin{align*}
    \epsilon^{(1 + |f|)\gamma}\sum_{k = 0}^{N - |f| - 1} 
        \epsilon^{k\alpha} \|R^{\mathbf{Y}, \mathbf{1}}\|_{N\gamma; I}^k
    & \lesssim \sum_{k = 0}^{N - |f| - 1} h(\epsilon)^{k} \epsilon^{(1 + |f|)\gamma - k\delta}
      \lesssim \epsilon^{\delta}\sum_{k = 0}^{N - |f| - 1} h(\epsilon)^{k}.
  \end{align*}
  On the other hand, since by Corollary~\ref{cor:partial-bound3}, 
  \(C^f_I \lesssim l(r)^{(1 - (N - |f|))\alpha + \delta}\) for each \(f \in \CF_{N - 1}\), we have 
  \begin{align*}
    & C_I^f \|R^{\mathbf{Y}, \mathbf{1}}\|_{N\gamma; I}^{N - |f|}
        \epsilon^{(N - 1)(N - |f|)\gamma} \epsilon^{\alpha + \gamma + \delta}\\
    \lesssim\ & h(\epsilon)^{N - |f|} \epsilon^{-(N - |f|)(\alpha + \delta) 
      - ((1 - (N - |f|))\alpha + \delta) + (N - 1)(N - |f|)\gamma + \alpha + \gamma + \delta}\\
    \lesssim\ & \epsilon^{(N - 1)(N - |f|)\gamma} h(\epsilon)^{N - |f|} \lesssim \epsilon^\delta h(\epsilon)^{N - |f|}
  \end{align*}
  Moreover, for any \(\tau = [h] \in \CT^+_{N - 1}\), by using Corollary~\ref{cor:partial-bound3} to 
  obtain the estimate \(\|\partial^\tau_X Y\|_{\infty; I} = \|\partial^h_X \phi(\mathbf{Y})\|_{\infty; I} 
  \lesssim l(r)^{\alpha + \delta}\), we obtain
  \begin{align*}
    & C_I^f \|\partial^\tau_X Y\|_{\infty; I}^{(N - |f|)} \epsilon^{(N - |f|)(|\tau| - 1)\gamma}\epsilon^{\alpha + \gamma + \delta}\\
    \lesssim\ & \epsilon^{-((1 - (N - |f|))\alpha + \delta) - (N - |f|)(\alpha + \delta) + 
      (N - |f|)(|\tau| - 1)\gamma + \alpha + \gamma + \delta} \lesssim 1.
  \end{align*}
  Thus, as \(\|\partial^f_X \phi(\mathbf{Y})\|_{\infty; I} \epsilon^{\alpha + \delta} \lesssim 1\) by 
  Corollary~\ref{cor:partial-bound3}, by substituting these estimates into Equation~\eqref{eq:h-est2}, we 
  conclude
  \[h(\epsilon) \lesssim 1 + \epsilon^\delta \sum_{k = 1}^N h(\epsilon)^k\]
  as claimed.
\end{proof}
\section{Case of bounded derivatives}
We conclude this article with a brief discussion on the difficulty in extending the argument of 
\cite{Lejay:12} (cf. also \cite{Caravenna:25}) to allow for linear growth in the coefficients of branched RDEs.
In the case where \(N = 2\), \cite{Lejay:12} provides a non-explosion principle for the
RDEs~\eqref{eq:rde-strat} in the case where the coefficient \(\phi\) is such that 
\[\|D\phi\|_\infty + \|D \partial^{\tree<X'>}\phi\|_\infty 
  = \|D\phi\|_\infty + \|D (D\phi \cdot \phi)\|_\infty < \infty.\]
We observe that the condition \(\|D \partial^{\tree<X'>}\phi\|_\infty < \infty\) suggests that 
\(D^2\phi\) decays at the same rate as the growth of \(\phi\). This is consistent with our observation 
in Section~\ref{sec:beyond}. With this result in mind, it is therefore natural to conjecture that for 
general \(N\), the bRDE~\eqref{eq:rde-strat} does not explode if
\[\sup_{f \in \CF_{N - 1}} \|D \partial^f_X \phi\|_\infty < \infty.\]
To obtain this result, the strategy of \cite{Lejay:12} is to define an Euler scheme 
\((Y^n_t)_{t \in 2^{-n}[2^n T]}\) for the RDE~\eqref{eq:rde-strat} for which one can show pre-compactness 
by using the Arzelà-Ascoli theorem by controlling 
\(\|R^n\|_{3\gamma} \lesssim \|\delta R^n\|_{3\gamma}\) with 
\[R^n_{st} := R^{Y^n, \mathbf{1}}_{st} = \delta(Y^n)_{st} - \phi(Y^n_s) \<\mathbf{X}_{st}, \tree<X'>\> 
- \partial^{\tree<X'>}_X \phi(Y^n_s) \<\mathbf{X}_{st}, \tree<1'>\>\]
where the inequality is due to a discrete sewing lemma (e.g. \cite[Theorem 1.18]{Caravenna:25}).
By observing the identity 
\[(\delta R^n)_{sut} = \sum_f (\delta(\partial^f_X Y^n)_{su} - R^{\phi(Y^n), f}_{su}) \<\mathbf{X}_{ut}, f\>,\]
it remains to estimate \(\|\delta(\partial^f_X Y^n) - R^{\phi(Y^n), f}\|_{(3 - |f|)\gamma}\). 
The only non-trivial case is \(f = \mathbf{1}\), which one can easily control via the 
identity
\begin{equation}\label{eq:remainder-identity}
  \begin{split}
    \int_0^1 D\phi(Y^\theta_{st}) \dd \theta R^{Y, \mathbf{1}}_{st} 
    =\ & \overbrace{\delta \phi(Y)_{st} - \sum_{\tau \in \CT^*_{N - 1}} \partial^\tau_X \phi(Y_s) \<\mathbf{X}_{st}, \tau\>}^{= R^{\phi(Y), \mathbf{1}}_{st}}\\
    & - \sum_{\tau \in \CT^*_{N - 1}} \int_0^1 (\partial^\tau_X \phi(Y^\theta_{st}) - \partial^\tau_X \phi(Y_s)) 
      \<\mathbf{X}_{st}, \tau\> \dd \theta\\
    & + \sum_{f \in \CF_{N - 2}} \int_0^1 D\phi(Y^\theta_{st})(\partial^f_X \phi(Y^\theta_{st}) - \partial^f_X \phi(Y_s))
      \<\mathbf{X}_{st}, [f]\> \dd \theta
  \end{split}
\end{equation}
where we denoted \(Y^\theta_{st} = (1 - \theta) Y_s + \theta Y_t\) for \(\theta \in [0, 1]\). Thus, 
assuming that \(\sup_f \|D\partial^f_X \phi\|_\infty < \infty\), \(\|R^n\|_{3\gamma}\) is controlled 
linearly by \(\|Y^n\|_\gamma\) for sufficiently small \(T\) uniformly in \(n\). On the other hand, 
it is easy to see that \(\|Y^n\|_\gamma\) is controlled linearly by \(\|R^n\|_{3\gamma}\) and thus,
one obtains a uniform estimate on \(\|R^n\|_{3\gamma}\) as desired. 

In the general case, the above two identities remain to hold and one can similarly define an Euler scheme 
\(Y^n\) for the bRDE~\eqref{eq:rde-strat} for which, by the discrete sewing lemma, one attempts to control 
\(\|\delta R^n\|_{(N + 1)\gamma}\). However, in applying \eqref{eq:remainder-identity}, one quickly 
realizes that higher level remainder terms introduce quadratic (or higher order) dependence on 
\(\|Y^n\|_\gamma\). To be more precise, while for \(N = 2\), \(\CF_{N - 1} = \CT_{N - 1}\) and 
\(\delta Y_{st} - R^{\phi(Y), \mathbf{1}}_{st}\) can be controlled by \eqref{eq:remainder-identity},
for \(N > 2\), the summand corresponding to \(f \in \CF_{N - 1} \setminus \CT_{N - 1}\) with 
\(\# f = 2\) can instead only be estimated via the second order Taylor remainder 
\(\int_0^1 \theta \int_0^1 D^2\phi(Y^{\theta\mu}_{st})[R^{Y, \mathbf{1}} \otimes R^{Y, \mathbf{1}}] \dd \mu \dd \theta\).
This introduces a quadratic dependence on \(\|Y^n\|_\gamma\) and thus, the previous argument in obtaining 
an a priori estimate fails.

\endappendix

\newcommand{\etalchar}[1]{$^{#1}$}

\end{document}